\documentclass[a4paper,reqno]{amsart}
\usepackage{latexsym,amsmath,amssymb,amsfonts}
\usepackage{color}

\usepackage{graphicx, caption, bbm}
\usepackage[usenames,dvipsnames]{pstricks}
\usepackage{float}
\usepackage{epsfig}
\usepackage{pst-grad} 
\usepackage{pst-plot} 

\newcommand{\R}{\mathbb R}
\newcommand{\N}{\mathbb N}
\newcommand{\ms}{\mathcal{MS}}

\newcommand{\meno}{\!\setminus\!}
\renewcommand{\div}{{\rm div}}
\newcommand{\dx}{\;\mathrm{d}x}
\newcommand{\h}{\mathcal{H}^1}
\renewcommand{\dh}{\;\mathrm{d}\mathcal{H}^1}
\newcommand{\dc}{\;\mathrm{d}\mathcal{H}^0}
\newcommand{\supp}{{\rm supp\,}}
\newcommand{\po}{\partial\Omega}
\newcommand{\pdo}{\partial_{D}\Omega}
\newcommand{\g}{\Gamma}
\renewcommand{\o}{\Omega}

\renewcommand{\ng}{\nabla_{\Gamma}}
\newcommand{\p}{\varphi}
\newcommand{\diffc}{\mathcal{D}_{\delta}(\o,U)}

\newcommand{\id}{\mathrm{Id}}

\theoremstyle{plain}
\begingroup
\newtheorem{theorem}{Theorem}[section]
\newtheorem{lemma}[theorem]{Lemma}
\newtheorem{proposition}[theorem]{Proposition}

\endgroup
\theoremstyle{definition}
\begingroup
\newtheorem{definition}[theorem]{Definition}
\newtheorem{remark}[theorem]{Remark}

\endgroup
\theoremstyle{remark}

\mathchardef\emptyset="001F

\numberwithin{equation}{section}

\title[Local minimality of triple points]{A second order local minimality criterion for the triple junction singularity of the Mumford-Shah functional.}
\author{R.\ Cristoferi}

\begin{document}

\maketitle

\begin{abstract}
This paper is the first part of an ongoing project aimed at providing a local minimality criterion, based on a second variation approach, for the triple point configurations of the Mumford-Shah functional.
\end{abstract}

\section{Introduction}

The importance of the Mumford-Shah functional (introduced in \cite{MS1, MS2} in the context of image segmentation) lies on the fact that it is a prototype for the class of variational problems that are commonly called \emph{free discontinuity problems}. These problems are characterized by a competition between volume and surface energy, and arise in many physical models (for example, in fracture mechanics).

The (homogeneous) \emph{Mumford-Shah functional} in the plane is defined as follows:
\begin{equation}\label{eq:ms}
\ms(u,\g):=\int_{\o\,\meno\,\g} |\nabla u|^2\dx + \h(\g\cap\o)\;,
\end{equation}
where $\Omega\subset\R^2$ is a $C^1$ domain, $\h$ denotes the $1$-dimensional Hausdorff measure, and $(u,\g)$ is a pair where $\g$ is a closed subset of $\R^2$ and $u\in H^1(\o\meno\g)$.

The existence of global minimizers in arbitrary dimension has been provided by De Giorgi, Carrieo and Leaci in \cite{DegCarLea}
(for other proof see, for instance, \cite{MadSol} and, for dimension $2$, \cite{DMMorSol, MorSol})
In the seminal paper \cite{MS2} it has been conjectured that, if $(u,\g)$ is a minimizing pair, then the set $\g$ is made by a finite union of $C^1$ arcs. Given this structure for granted, it is not difficult to prove (see \cite{MS2}) that the only possible singularities of the set $\g$ can be of the following two types:
\begin{itemize}
\item $\g$ ends at an interior point (the so called \emph{crack-tip}),
\item three regular arcs $\g^1, \g^2, \g^3$ meeting at an interior point $x_0$ with equal angles of $2\pi/3$ (the so called \emph{triple point}).
\end{itemize}
Although several results on the regularity of the discountinuity set $\g$ have been obtained (but we will not recall them here), the conjecture is still open.\\

The deep lack of convexity of the functional \eqref{eq:ms} naturally leads one to ask what conditions imply that critical configurations as above are local minimizers. The study of such a conditions has been initiated by Cagnetti, Mora and Morini in \cite{CagMorMor}, where they deal with the regular part of the discontinuity set. In particular they introduce a suitable notion of second variation and  prove that the strict positivity of the associated quadratic form  is a sufficient condition for the local minimality with respect to small $C^2$ perturbations of the discontinuity set $\g$.
Subsequently, the above result has been strongly improved by Bonacini and Morini in \cite{BonMor}, where it is shown that  if $(u, \Gamma)$ is a critical pair for \eqref{eq:ms} with strictly positive second variation, then it locally minimizes the functional with respect to small $L^1$-perturbations of $u$, namely there exists $\delta>0$ such that 
$$
MS(u, \Gamma)<MS(v, \Gamma')
$$
for all admissible pairs $(v,\Gamma')$ satisfying $0<\|u-v\|_{L^1}\leq \delta$.

Among other results on local and global minimality criteria, we would like to recall the important work \cite{AlbBouDM} of Alberti, Bouchitt\'{e} and
Dal Maso, where they introduce a general calibration method for a family of non convex variational problems. In particular they apply this method to the case of the Mumford-Shah functional to obtain minimality results for some particular configurations.
Moreover, Mora in \cite{Mora} used that calibration technique to prove that a critical configuration $(u,\g)$, where $\g$ is made by three line segments meeting at the origin with equal angles, is a minimizer of the Mumford-Shah energy in a suitable neighborhood of the origin, with respect to its Dirichlet boundary conditions.
Finally, we recall that the same method has been used by Mora and Morini in \cite{MorMor}, and by Morini in \cite{Mor} (in the case of the non homogeneous Mumford-Shah functional), to obtain local and global minimality results in the case of a regular curve $\g$.\\

Our aim is to continue the investigation of second order sufficient conditions, by considering for the first time the case of a singular configuration (the triple point configuration).
For the area functional, a general approach to treat the case of singularities appears for the first time in the works by Cicalese, Leonardi and Maggi (see \cite{CicLeoMag} and \cite{CicLeoMag2} for the application to the stability of the planar double bubble), and in the case of higher dimensions by Leonardi and Maggi (see\cite{CicLeoMag3}).

The plan is the following. In Section \ref{sec:variations} we compute, as in \cite{CagMorMor}, the second variation of the functional $\ms$ at a triple point configuration $(u,\g)$, with respect to a one-parameter family of (sufficiently regular) diffeomorphisms $(\Phi_t)_{t\in(-1,1)}$, where each $\Phi_t$  equals the identity in the part of $\po$ where we impose the Dirichlet condition and $\Phi_0=\id$. 
The idea is then to consider  for each time $t\in(-1,1)$ the pair $(u_t, \g_t)$, where $\g_t:=\Phi_t(\g)$, and $u_t\in H^1(\o\meno \g_t)$ minimizes the Dirichlet energy with respect to the given boundary conditions.
.
We show that the second variation can be written as follows:
\begin{equation}\label{secvarintro}
{\frac{\mathrm{d}^2}{\mathrm{d}t^2} {\ms(u_t,\g_t)}}_{|t=0}=
                     \partial^2\ms(u,\g)\bigl[ (X\cdot\nu^1,X\cdot\nu^2,X\cdot\nu^3) \bigr]+R\,,
\end{equation}
where $\partial^2\ms(u,\g)$ is a nonlocal (explicitly given) quadratic form, $X$ is the velocity field at time $0$ of the flow $t\mapsto \Phi_t$ (see Definition \ref{def:family}), and $\nu^i$ is the normal vector field on $\g^i$. Moreover, the remainder  $R$ vanishes whenever $(u,\g)$ is a critical triple point. Thus, in particular, if $(u,\g)$ is a local minimizer with respect to smooth perturbations of $\g$, then  the quadratic form $\partial^2\ms(u,\g)$ has to be non-negative. 

Next we address the question as to whether the strict positivity of $\partial^2\ms(u,\g)$, with $(u, \g)$ critical, is a sufficient condition for local minimality. The main result (see Theorem~\ref{thm:minC2}) is the following:  if $(u,\g)$ is a strictly stable critical pair, then there exists $\delta>0$ such that 
$$
\ms(u,\g)<\ms(v,\Phi(\g))\,,
$$
for any $W^{2,\infty}$-diffeomorphism $\Phi:\bar{\o}\rightarrow\bar{\o}$ and any function $v\in H^1\bigl(\o\meno\Phi(\g)\bigr)$ satisfying  the proper boundary conditions, provided that $\|\Phi-\id\|_{W^{2,\infty}}\leq \delta$, and $\Phi(\g)\neq\g$.
The above result can be seen as the analog for triple points configurations of the minimality result established in \cite{CagMorMor} in the case of regular discontinuity sets. 

From the technical point of view the presence of the singularity makes the problem considerably more challenging. 
The main  difficulty lies in the construction of a suitable family of bijections $(\Phi_t)_{t\in[0,1]}$ connecting the critical triple point configuration with the competitor, in such a way that the ``tangential'' part along $\Phi_t(\Gamma)$ of the velocity field $X_t$  of the flow $t\mapsto \Phi_t$ is controlled by its normal part. Moreover, one also has to make sure that the $C^2$-closeness to the identity is preserved along the way. 
This turns out to be a challenging task, due to the presence of the triple junction which poses nontrivial regularity problems.
This technical difficulty has been first addressed in \cite{CicLeoMag}, where the authors solved the problem of \emph{re-parametrizing} a diffeomorphism between two curves in order to obtain a control of the tangential part of the new diffeomorphism on the curve with its normal one. Such a control is fundamental when one aims at using a second variation approach. In our case, the presence of the volume term prevent us to use directly the result in \cite{CicLeoMag}, but forces us to use the strategy described below (formula \eqref{connection}), for which different estimates, not explicitly written in the above work, are needed. For reader's convenience, we present here a construction that is similar in spirit, but independent, to the one provided by Cicalese, Leonardi and Maggi, where the relevant features that allow to obtain the estimates are explicitly pointed out.
Once such a construction is performed, one proceeds in the following way. Let $g(t):=\ms(u_t,\g_t)$ and notice that, by criticality, we have $g'(0)=0$. Thus, recalling \eqref{secvarintro}, it is possible to write
\begin{align}
\ms(v,\Phi(\g))\,-&\,\ms(u,\g) = \int_0^1 (1-t)g''(t)\,\mathrm{d}t\nonumber\\
=& \int_0^1 (1-t)\bigl(\partial^2\ms(u_t,\g_t)[X_t\cdot\nu_t]+R_t \bigr) \,\mathrm{d} t.\label{connection}
\end{align}
If $\g_t$ is sufficiently $C^2$-close to $\g$, by the strict positivity assumption on $\partial^2\ms(u,\g)$, we may conclude by continuity that 
$$
\partial^2\ms(u_t,\g_t)[X_t\cdot\nu_t]\geq C\|X_t\cdot\nu_t\|^2\,.
$$
Unfortunately,  the remainder $R_t$ depends also on the tangential part of $X_t$. However, if  the family $(\Phi_t)_t$ is properly constructed, on can ensure that such a tangential part is controlled by $X_t\cdot\nu_t$ and  
$$
|R_t|\leq \varepsilon\|X_t\cdot\nu_t\|^2
$$
for any $\varepsilon>0$, provided that the $\Phi_t$'s are sufficiently $C^2$-close to the identity. Plugging the above two estimates into \eqref{connection} one eventually concludes that, for a $\Phi$'s satisfying the above assumptions, $\ms(v,\Phi(\g))>\ms(u,\g)$. 

We conclude this introduction by observing that the above result represents just the first step of a more general strategy aimed at 
establishing the local minimality with respect to the $L^1$-topology in the spirit of \cite{BonMor}, which will be the subject of future investigations.\\


\section{Setting}

Here we collect the terminology and we introduce all the objects we will need in the rest of the paper.
First of all, we need to specify the class of triple points we are interested in.

\begin{definition}\label{def:triple}
We say that a pair $(u,\g)$ is a \emph{(regular admissible) triple point} if
\begin{itemize}
\item $\Gamma=\{x_0\}\cup\Gamma^1\cup\Gamma^2\cup\Gamma^3$, where $x_0\in \Omega$ and the $\Gamma^i$'s are three disjoint simple open\footnote{By an \emph{open} curve we mean a curve $C:I\rightarrow\R^2$, where $I\subset\R$ is an open interval.} curves in $\o$ that are of class $C^3$ and $C^{2,\alpha}$ up to their closure. We also suppose $\partial\g_i=\{x_0, x_i\}$, where $x_0\in\o$ and $x_i\in\partial\o$ with $x_i\neq x_j$ for $i\neq j$,
\item denoting by $\nu_i$ the normal vector to $\g_i$, we require the angle between $\nu_i(x_0)$ and $\nu_{j}(x_0)$ to be less than or equal to $\pi$, for all $i\neq j$ (for the choice of the parametrization on each $\gamma_i$, see Section \ref{sec:prel}),
\item each $\bar{\g}^i$ to do not intersect $\partial\o$ tangentially,
\item there exists $\pdo\subset\subset\po\meno\g$, relatively open in $\po$, such that $u$ solves
\begin{equation}\label{eq:u}
\int_{\Omega\,\meno\,\Gamma} \nabla u \cdot \nabla z\dx=0\,,
\end{equation}
for every $z\in H^1(\o\meno\g)$ with $z=0$ on $\pdo$.
\end{itemize}
\end{definition}

\begin{remark}
The regularity we impose on the curves $\g^i$'s is not so restrictive as it may seem: indeed we will work with critical triple points (see Definition
~\ref{def:crit}), and it was proved in \cite{KocLeoMor} that, for critical configurations, each $\g_i$ is analytic as soon as it is of class $C^{1,\alpha}$, and the regularity theory tells us that each curve is of class $C^{2,\alpha}$ up to its closure.
We would like to point out that the assumption that each curve $\g_i$ is open has been made just for convenience, and does not prevent the use of $(u,\g)$, with $u\in H^1(\o\meno\g)$, as an admissible pair at which to compute the functional $\ms$.
\end{remark}

The last condition in Definition \ref{def:triple} tells us that $u$ is a weak solution of
$$
\left\{
\begin{array}{ll}
\triangle v = 0 & \text{ in } \o\meno\bar{\g}\,,\\
v = u & \text{ on } \pdo\,,\\
\partial_{\nu_{\partial\o}} v = 0 & \text{ on } \partial\o\meno\pdo\,,\\
\partial_{\nu} v = 0 & \text{ on } \bar{\g}\,.\\
\end{array}
\right.
$$
From the results on elliptic problems in domains with corners (see, \emph{e.g.}, \cite{Gris}) and from the regularity of $\partial\o$, we know that $u$ can have a singularity near $\mathcal{S}$, the relative boundary of $\pdo$ in $\partial\o$: namely, $u$ can be $H^1$ but not $H^2$ in a neighborhood of $\mathcal{S}$.
Thus, the gradient of $u$ may not be bounded in that region. In a future application of the present work we will need to impose a bound on the $L^\infty$-norm of the gradient of the admissible competitors. But this can be done only far from $\mathcal{S}$. So, we are forced to consider competitors equals to $u$ in a neighborhood of $\mathcal{S}$.

\begin{definition}
Given a regular triple point $(u,\g)$, we say that an open set $U\subset\Omega$ is an \emph{admissible subdomain} if $\g\subset U$ and $\overline{U}\cap S=\emptyset$.
In this case we define
$$
\ms\bigl((u,\g);U\bigr) := \int_{U\,\meno\,\g} |\nabla u|^2\dx + \h(\g)\;.
$$
Moreover, given an open set $A\subset\o$, we denote by $H^1_U(A)$ the space of functions $z\in H^1(A)$ such that $z=0$ on $(\o\meno U)\cup \pdo$.
\end{definition}

\noindent
\textbf{Notation:} we will call $\o^1, \o^2, \o^3$ the three open connected components of $\Omega\meno\bar{\g}$.\\

Our strategy requires to perform the first and the second variation of our functional $\ms$. So, we need to specify the perturbations of the set $\g$ and of the function $u$ we want to consider.

\begin{definition}\label{def:family}
Let $(u,\g)$ be a triple point and let $U$ be an admissible subdomain.
We say that a family of bijections of $\overline{\o}$ onto itself, $(\Phi_t)_{t\in(-1,1)}$, is \emph{admissible} for $(u,\g)$ in $U$ if
the following conditions are satisfied:
\begin{itemize}
\item $\Phi_0$ is the identity map $\id$,
\item $\Phi_t=\id$ in $(\o\meno U)\cup\pdo$, for each $t\in(-1,1)$,
\item $\Phi^i_t:=(\Phi_t)_{|\o^i}$ is a diffeomorphism of class $C^1$, for each $t\in(-1,1)$,
\item $\Phi_t$ is of class $C^3$ on $\g$, for each $t\in(-1,1)$,
\item for each $x\in\bar{\o}$, the map $t\rightarrow\Phi_t(x)$ is of class $C^2$.
\end{itemize}
In this case, we define:
$$
X_{\Phi_t}:=\dot{\Phi}_t\circ\Phi_t^{-1}\,,\quad\quad\quad Z_{\Phi_t}:=\ddot{\Phi}_t\circ\Phi_t^{-1}\,,
$$
where with $\dot{\Phi}_t$ we denote the derivative with respect to the variable $s$ of the map $(s,x)\rightarrow \Phi_s(x)$ computed at $(t,x)$. Notice that the above objects are well defined.
Moreover we also introduce the following abbreviations
$$
X_t:=X_{\Phi_t}\,,\quad\quad Z_t:=Z_{\Phi_t}\,,\quad\quad X:=X_0\,,\quad\quad Z:=Z_0\,,
$$
where no risk of confusion can occur.
\end{definition}

\begin{remark}
Usually the variations that are considered are $C^3$ diffeomorphisms of $\bar{\o}$ for every fixed $t$, and functions of class $C^2$ for every fixed $x$. The reason why we need to consider this weaker class of admissible functions is because in the construction we will provide in Proposition \ref{prop:x}, the regularity we will be able to prove is the one of the above definition. However, the above hypotheses on
$(\Phi_t)_t$ suffice to be able to compute the first and the second variation of the functional $\ms$.
\end{remark}

The above variations will affect only the set $\g$, \emph{i.e.}, at every time $t$ we will consider the set $\g_t:=\Phi_t(\g)$. Since our functional depends also on a function $u$, we have to choose, for each time $t$, a suitable function $u_t$ related to the set $\g_t$ at which compute our functional $\ms$. The idea, as in \cite{CagMorMor}, is to choose the function that minimizes the Dirichlet energy.

\begin{definition}\label{def:umod}
Let $\Phi:\bar{\o}\rightarrow\bar{\o}$ be a diffeomorphism such that $\Phi=\id$ on $(\o\meno U)\cup\pdo$, and set $\g_\Phi:=\Phi(\g)$. We define $u_{\Phi}$ as the unique solution of:
$$
\left\{
\begin{array}{ll}
\int_{\o\,\meno\,\g_{\Phi}} \nabla u_{\Phi} \cdot \nabla z\dx = 0 & \mbox{for each } z\in H^1_U(\o\meno\g_\Phi)\,, \\
u_{\Phi} = u & \mbox{in } (\o\meno U)\cup \pdo\,, \\
u_{\Phi} \in H^1(\o\meno\g_\Phi)\,.
\end{array}
\right.
$$
Moreover, given a family of admissible diffeomorphisms $(\Phi_t)_t$, we set $u_t:=u_{\Phi_t}$, and we define the function $\dot{u}_t(x)$ as the derivative with respect to the variable $s$ of the map $(s,x)\mapsto u_s(x)$, computed in $(t,x)$. For simplicity, set $\dot{u}:=\dot{u}_{0}$. 
\end{definition}

We are now in position to describe the admissible variations.

\begin{definition}\label{def:var}
We define the \emph{first} and the \emph{second variation} of the functional $\ms$ at a regular admissible triple point $(u,\g)$ in $U$, with respect to the family of admissible diffeomorphisms $(\Phi_t)_{t\in(-1,1)}$, as
$$
\frac{\mathrm{d}}{\mathrm{d}t} {\ms\bigl((u_t,\g_t);U\bigr)}_{|t=0}\,,\quad\quad\quad \frac{\mathrm{d}^2}{\mathrm{d}t^2} {\ms\bigl((u_t,\g_t);U\bigr)}_{|t=0}\,,
$$
respectively.
\end{definition}


\section{Preliminary results}\label{sec:prel}

\subsection{Geometric preliminaries} We collect here some geometric definitions and identities that will be useful later.
First of all, we will use the following matrix notation: if $A:\R^2\rightarrow\R^2$ and $v_1,v_2\in\R^2$, we set
$$
A[v_1,v_2]:=A[v_1]\cdot v_2\,.
$$

Let $\gamma\subset\R^2$ be a curve of class $C^2$ and let $\tau:\gamma\rightarrow\mathbb{S}^1$ be the tangent vector field on $\gamma$. Given an orientation on $\gamma$ it is possible to define a signed distance function from $\gamma$ as follows:
$$
\mathrm{d}_\gamma(x+t\nu(x)):=t\,,
$$
where $\nu(x)$ is the normal vector to $\gamma$ at the point $x$, obtained by rotating $\tau(x)$ counterclockwise. This signed distance turns out to be of class $C^2$ in a tubular neighborhood
$\mathcal{U}$ of $\gamma$; moreover, its gradient coincides with $\nu$ on $\gamma$. In the following we will use the extension of the normal vector field given by the gradient of the signed distance from $\gamma$, that we will set $\nu:\mathcal{U}\rightarrow\mathbb{S}^1$.  

Given a smooth vector field $g:\mathcal{U}\rightarrow\R^k$, we define the \emph{tangential differential} $D_{\gamma}g(x)$ at a point $x\in\gamma$ ($\nabla_\gamma g(x)$ if $k=1$) by
$D_{\gamma}g(x):=dg(x)\circ\pi_x$, where $dg(x)$ is the classical differential of $g$ at $x$ and $\pi_x$ is the orthogonal projection on $T_x\gamma$, the tangent line to $\gamma$ at $x$. If $g:\mathcal{U}\rightarrow\R^2$ we define its \emph{tangential divergence} as $\div_{\gamma}g:=\tau\cdot\partial_{\tau}g$.

We define the \emph{curvature} of $\gamma$ as the function $H:\mathcal{U}\rightarrow\R$ given by $H:=\div\nu$. Notice that, since $\partial_{\nu}\nu=0$ on $\g$, we can write $H=\div_{\gamma}\nu=D\nu[\tau,\tau]$. \\
For every smooth vector field $g:\mathcal{U}\rightarrow\R^2$ the following \emph{divergence formula holds}:
\begin{equation}\label{eq:curv}
\int_{\gamma} \div_{\gamma}g\dh = \int_{\gamma} H(g\cdot\nu)\dh+\int_{\partial\gamma}g\cdot\eta\dc\,,
\end{equation}
where $\eta$ is a unit tangent vector pointing out of $\gamma$ in each point of $\partial\gamma$.
Moreover, if $\Phi:\mathcal{U}\rightarrow\mathcal{U}$ is an orientation preserving diffeomorphism, and we denote by $\gamma_{\Phi}:=\Phi(\gamma)$, a possible choice for the orientation  of $\gamma_{\Phi}$ is given by:
$$
\nu_{\Phi}:=\frac{(D\Phi)^{-T}[\nu]}{\bigl| (D\Phi)^{-T}[\nu] \bigr|}\circ\Phi^{-1}\,.
$$
In this case, the vector $\eta$ of the divergence formula \eqref{eq:curv} becomes
$$
\eta_{\Phi}:=\frac{D\Phi[\eta]}{\bigl| D\Phi[\eta] \bigr|}\circ\Phi^{-1}\,.
$$
In particular, for an admissible flow $(\Phi_t)_{t\in(-1,1)}$, we will use the following notation: $\nu_t:=\nu_{\Phi_t}$, $\eta_t:=\eta_{\Phi_t}$, and we will denote by $H_t$ the curvature of $\gamma_t$.

Finally, setting $J_{\Phi}:=\bigl| (D\Phi)^{-T}[\nu] \bigr|\det D\Phi$, for every $f\in L^{1}(\gamma_{\Phi})$ the following \emph{area formula} holds (see \cite[Theorem~2.91]{AFP}):
$$
\int_{\gamma_{\Phi}} f\dh = \int_{\gamma} (f\circ\Phi)\, J_{\Phi}\dh\,.
$$
We now treat triple points. Fix for $\partial\o$ the clockwise orientation and orient the curves $\g_i$'s in such a way that $\nu^i(x_i)=\tau_{\partial\o}(x_i)$ for each $i=1,2,3$ (see Figure ~\ref{fig:triple}), where $\nu^i$ is the normal vector on $\g_i$.

\begin{figure}[H]
\includegraphics[scale=0.7]{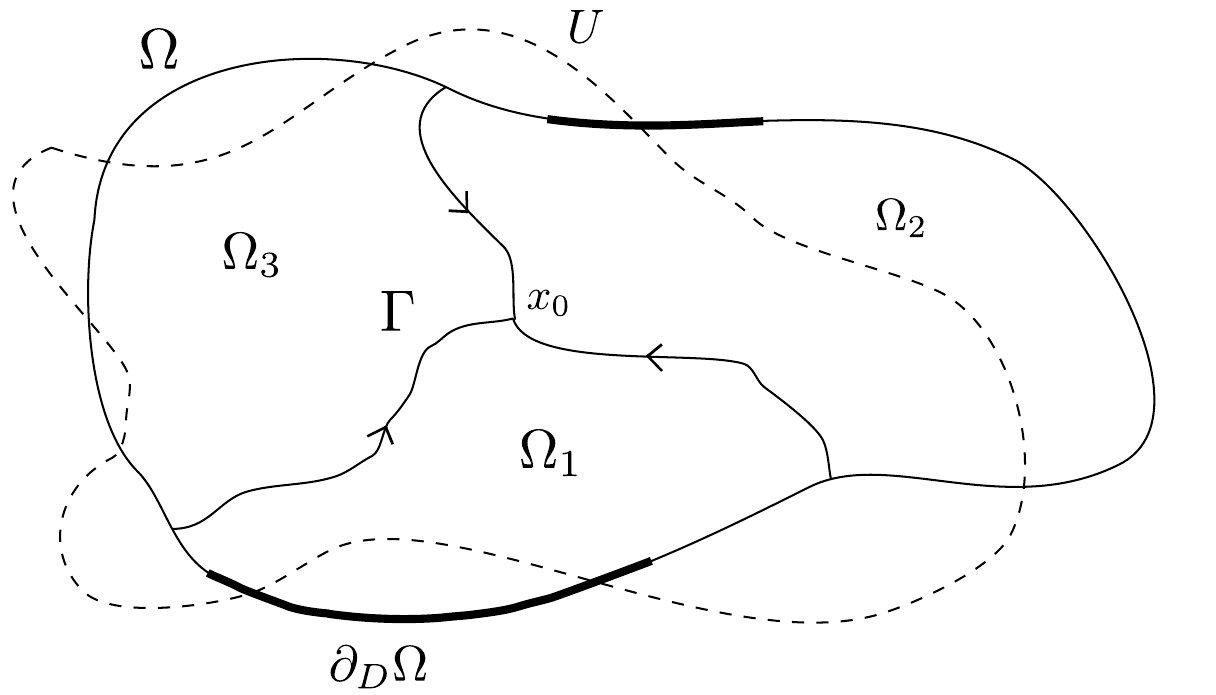}
\caption{An admissible triple point with the chosen orientation and an admissible subdomain $U$. The bold part of the boundary represents the set $\partial_D\Omega$.}
\label{fig:triple}
\end{figure}

For the sake of simplicity we will use the following notation: given $f:\Gamma\rightarrow\R^k$, we will denote by $f_i$ its restriction to $\Gamma_i$, and we will write
$$
\int_{\partial\Gamma} f\dc:=\sum_{i=1}^3 \bigl( f_i(x_0) + f_i(x_i) \bigr)\,.
$$

In the following we will also need to use the trace of a function on $\g$. We recall that, since each $\g_i$ is open, $x_i\not\in\g_i$, for each $i=0,1,2,3$.

\begin{definition}
Let $\g$ be a regular admissible triple point, and let $z\in H^1(\o\meno\bar{\g})$. We define the \emph{traces} $z^+, z^-$
of $z$ on $\g$ as follows: let $x\in\g$ and define
$$
z^{\pm}(x):=\lim_{r\rightarrow0^+}\frac{1}{\bigl|  B_r(x)\cap V_x^{\pm} \bigr|}\int_{B_r(x)\cap V_x^{\pm}} z(y)\,\mathrm{d}y\,,
$$
where $V_x^{\pm}:=\{ y\in\R^2 \,:\, \pm(y-x)\cdot\nu^i(x)\geq0 \}$, if $x\in\g_i$.\\
\end{definition}

In the computation of the second variation we will need some geometric identities, that we collect in the following lemma.
The proofs of the first block of identities are the same as those of \cite[Lemma~3.8]{CagMorMor}, and hence we will not repeat them here.
We just need to prove the last three.

\begin{lemma}\label{lem:geom}
The following identities hold on each $\g^i$:
\begin{enumerate}
\item $D^2 u^{\pm}[\nu^i,\nu^i] = -\triangle_{\g_i}u^{\pm}$;
\item $D^2 u^{\pm}[X,\nu^i] = -(X\cdot\nu^i)\triangle_{\g_i}u^{\pm}-D\nu^i[\nabla_{\g^i} u^{\pm}, X]$;
\item $\div_{\g_i}[(X\cdot\nu^i)\nabla_{\g^i} u^{\pm}]=(D_{\g^i}X)^{T}[\nu^i,\nabla_{\g^i} u^{\pm}]-\nabla^2 u^{\pm}[X,\nu^i]$;
\item $\partial_{\nu^i}H^i=-|D\nu^i|^2=-(H^i)^2$;
\item $D^2 u^{\pm}[\nu^i,\nabla_{\g^i} u^{\pm}]=-D\nu^i[\nabla_{\g^i} u^{\pm}, \nabla_{\g^i} u^{\pm}]=-H_i|\nabla_{\g_i} u^{\pm}|^2$;
\item $\dot{\nu}^i=-(D_{\g^i}X)^{T}[\nu^i]-D_{\g^i}\nu^i[X]=-\nabla_{\g}(X\cdot\nu)$;
\item $\frac{\partial}{\partial t}{\bigl(\dot{\Phi}_t\cdot(\nu^i_t\circ\Phi_t)J_{\Phi_t}\bigr)}_{|t=0}= Z\cdot\nu^i-2X^{||}\cdot\nabla_{\g^i}(X\cdot\nu^i)+D\nu^i[X^{||},X^{||}]+\div_{\g^i}((X\cdot\nu^i)X)$.
\end{enumerate}
Moreover, the following identities are satisfied:
\begin{enumerate}
\item[(i)] $\frac{\partial}{\partial t}{(\eta^i_t\circ\Phi_t)}_{|t=0}=(D_{\g_i}X)^{T}[\nu^i,\eta^i]\nu^i$, on $\partial\g^i$;
\item[(ii)] $X\cdot\frac{\partial}{\partial t}{(\eta^i_t\circ\Phi_t)}_{|t=0}=-(X\cdot\nu^i)\dot{\nu}^i\cdot\eta^i-H^i(X\cdot\nu^i)(X\cdot\eta^i)$, on $\partial\g^i$;
\item[(iii)] $Z\cdot\nu_{\po}+D\nu_{\po}[X,X]=0$ on $\partial\g^i\cap\po$.
\end{enumerate}
\end{lemma}

\begin{proof}
\emph{Proof of (i)}. Let $w_t:=D\Phi_t(x)[\tau(x)]$. Then
$$
\frac{\partial}{\partial t}{(\eta^i_t\circ\Phi_t)}_{|t=0}=\frac{\partial}{\partial t}\frac{w_t}{|w_t|}\,.
$$
Since $\dot{w}_0=D_{\g^i}\dot{\Phi}[\tau^i]= D_{\g^i}X[\tau^i]D_{\g^i}\tau$, we obtain
$$
\frac{\partial}{\partial t}{(\eta^i_t\circ\Phi_t)}_{|t=0}=D_{\g^i}X[\tau]-(D_{\g^i}X)^T[\tau,\tau]\,,
$$
we conclude.\\
\emph{Proof of (ii)}. This identity follows by taking the scalar product of identity (6) with $(X\cdot\nu)\eta$, and by using (i).\\
\emph{Proof of (iii)}. This one follows by deriving with respect to the time the identity
$$
(X_t\circ\Phi_t)\cdot(\nu_{\po}\circ\Phi_t)=0\,,
$$
that holds on $\partial\g^i\cap\po$.
\end{proof}


\subsection{Properties of the function $\dot{u}$} \label{sec:dotu}

In the computations of the first and the second variation we need to know some properties of the family of functions $(u_t)_t$ that we state here. First of all we need to prove that the function $\dot{u}$ actually exists. This is provided by the following result, whose proof is just the same as those of \cite[Proposition 8.1]{CagMorMor}, where the elliptic estimates in $W^{2,p}$ for $p<4$, needed to prove the second part are, in our case, provided by Theorem ~\ref{thm:regu}.

\begin{proposition}
Let $(\Phi_t)_t$ be an admissible family of diffeomorphisms, and let $(u_t)_t$ be the functions defined in Definition ~\ref{def:umod}. Set $\widetilde{u}_t:=u_t\circ\Phi_t$ and $v_t:=\widetilde{u}_t-u$. Then the following properties hold true:
\begin{itemize}
\item[(i)] the map $t\mapsto v_t$ belongs to $C^1\bigl( (-1,1); H^1_U(\o\meno\g) \bigr)$;
\item[(ii)] for every $\bar{x}\in\bar{\g}$, let $B$ be a ball centered in $\bar{x}$ such that $B\meno\g$ has two (or, if $\bar{x}=x_0$, three) connected components $B_1$, $B_2$ (and $B_3$). For every $t\in(-1,1)$, let $\widetilde{u}^i_t$ be the restriction of $\widetilde{u}_t$ to $B_i$. Then we have that the map $\widehat{u}^i(t,x):=\widetilde{u}^i_t(x)$ belongs to $C^1\bigl( (-1,1)\times \bar{B}_i \bigr)$.
\end{itemize}
\end{proposition}

Using the above proposition it is possible to prove the following result, whose proof is just the same as those of \cite[(3.6) of Theorem 3.6]{CagMorMor}.

\begin{proposition}\label{prop:pru}
The function $\dot{u}$ exists, it is a well defined function of $H^1_U(\o\meno\g)$. 
Moreover, it is harmonic in $\o\meno\bar{\g}$ and satisfies the following Neumann boundary conditions:
$$
\partial_{\nu_{\partial\o}} \dot{u} = 0\;\;\;\;\;\text{ on } (\partial\o\meno\pdo)\cap U\,,
$$
\begin{equation}\label{eq:part}
\partial_{\nu}\dot{u}^{\pm}=\div_{\g}\bigl((X\cdot\nu)\ng u^{\pm}\bigr)\,\quad\mbox{on } \g\,.
\end{equation}
In particular, the following equation holds:
\begin{equation}\label{eq:dotu}
\int_{\o} \nabla\dot{u}\cdot\nabla z\dx = \int_{\g} \bigl[ \div_{\g}\bigl((X\cdot\nu)\ng u^+\bigr)z^+ - \div_{\g}\bigl((X\cdot\nu)\ng u^-\bigr)z^- \bigr] \dh\,,
\end{equation}
for each $z\in H^1_U(\o\meno\g)$.\\
\end{proposition}

\begin{remark}
First of all we notice that the right-hand side of \eqref{eq:part} is well defined. Indeed, by Theorem ~\ref{thm:regu} that $u$ is of class $H^2$ in a neighborhood of $\g$, and thus $\ng u^\pm\in H^{\frac{1}{2}}(\bar{\g})$. So, since $\g$ and $X$ are regular, we get that
$(X\cdot\nu)\ng u^\pm\in H^{\frac{1}{2}}(\bar{\g})$.
\end{remark}


\section{First and second variation}\label{sec:variations}

The aim of this section is to compute the first and the second variation of the functional $\ms$ at a triple point $(u,\g)$.

\begin{theorem}\label{thm:secvar}
Let $(u,\g)$ be a triple point, $U$ an admissible subdomain and $(\Phi_t)_{t\in(-1,1)}$ an admissible family for
$(u,\g)$ in $U$. Set $f:=|\nabla_{\g} u^-|^2 - |\nabla_{\g} u^+|^2 + H$.
Then the first variation of the functional $\ms$ computed at $(u,\g)$ with respect to $(\Phi_t)_{t\in(-1,1)}$ is given by:
\begin{equation}\label{eq:firstvar}
\frac{\mathrm{d}}{\mathrm{d} t}{\ms\bigl((u_t,\g_t);U\bigr)}_{|t=0} = \int_{\g} f(X\cdot\nu)\dh + \int_{\partial{\g}} X\cdot\eta\dc\,,
\end{equation}

while the second variation reads as:
\begin{align}\label{eq:secvar}
\frac{\mathrm{d}^2}{\mathrm{d}t^2}&{\ms\bigl((u_t,\g_t);U\bigr)}_{|t=0} = -2\int_{U}|\nabla\dot{u}|^2\dx + \int_{\g}|\ng(X\cdot\nu)|^2\dh + \int_{\g}H^2(X\cdot\nu)^2\dh \nonumber \\
& +\int_{\g}f\bigl[ Z\cdot\nu-2X^{||}\cdot\ng(X\cdot\nu)+D\nu[X^{||},X^{||}] - H(X\cdot\nu)^2 \bigr]\dh + \int_{\partial\g} Z\cdot\eta \dc \,.
\end{align}
\end{theorem}

\begin{proof}
\emph{Computation of the first variation}. In order to derive the function
$$
t\mapsto\ms\bigl((u_t,\g_t);U\bigr)=\int_{U\,\meno\,\g_t} |\nabla u_t|^2\dx + \h(\g_t)\,,
$$
we treat the two terms separately. For the first one we write
\[
\int_{U\,\meno\,\g_t} |\nabla u_t|^2\dx = \sum_{i=1}^3
\int_{U\cap\o^i_t} |\nabla u_t|^2\dx\,,
\]
where $\o^i_t:=\Phi_t(\o^i)$. By our assumptions, $\o_t^i$ is diffeomorphic to $\o_i$ through $\Phi^i_t$ and thus we can apply the change of variable formula. So we have
\begin{align*}
\frac{\mathrm{d}}{\mathrm{d} s}&\Biggl( \int_{U\,\meno\,\g_s}|\nabla u_s|^2\dx \Biggr)_{|s=t} =
				\frac{\mathrm{d}}{\mathrm{d} s}\Biggl( \int_{U\,\meno\,\g}|\nabla u_s\circ\Phi_s|^2\det D\Phi_s \dx\Biggr)_{|s=t} \\
&= \int_{U\,\meno\,\g} \Bigl[ 2(\nabla u_t\circ\Phi_t)\cdot\bigl( (\nabla \dot{u}_t\circ\Phi_t) + (D^2u_t\circ\Phi_t)\dot{\Phi}_t \bigr) 
  + |\nabla u_t\circ\Phi_t|^2 \div X_t\circ\Phi_t\Bigr]\det D\Phi_t  \dx \\
&= 2\int_{U\,\meno\,\g_t} \nabla u_t\cdot\nabla\dot{u}_t\dx + \int_{U\,\meno\,\g_t} \Bigl( 2D^2 u_t[\nabla u_t, X_t] + |\nabla u_t|^2\div X_t \Bigr)\dx\,.
\end{align*}
Recalling that $\dot{u}_t\in H^1_U(U\meno\g_t)$ by Proposition ~\ref{prop:pru}, from \eqref{eq:u} we get that the first integral vanishes. Moreover, since it is possible to write
$$
2D^2 u_t[\nabla u_t, X_t] + |\nabla u_t|^2\div X_t = \div\bigl(|\nabla u_t|^2 X_t \bigr)\,,
$$
integrating by parts in each connected component of $\o\meno\g_t$, and recalling that $X_t\cdot\nu_{\po}=0$, we get
$$
\frac{\mathrm{d}}{\mathrm{d} s}\Biggl( \int_{U\,\meno\,\g_s}|\nabla u_s|^2\dx \Biggr)_{|s=t} =\int_{\g_t} \bigl( |\nabla u_t^-|^2 - |\nabla u_t^+|^2 \bigr) (X_t\cdot\nu_t)\dh\,.
$$
Finally we also notice that in the last expression, we can substitute the operator $\nabla$ with
the operator $\nabla_{\g_t}$, since $\partial_{\nu_t}u_t=0$.

For the second term, it is well known (see, \emph{e.g.}, \cite{Sim}) that 
\begin{eqnarray*}
\frac{\mathrm{d}}{\mathrm{d} t}\bigg( \h(\g_s) \biggr)_{|s=t} &=& \int_{\g_t} \div_{\g_t} X_t\dh = \int_{\g_t} H_t (X_t\cdot\nu_t)\dh + \int_{\partial{\g_t}} X_t\cdot\eta_t\dc\,.
\end{eqnarray*}
Hence, defining the function $f_t$ on $\g_t$ as $f_t:=|\nabla_{\g_t} u_t^-|^2 - |\nabla_{\g_t} u_t^+|^2 + H_t$, we obtain
\begin{equation}\label{eq:firvar}
\frac{\mathrm{d}}{\mathrm{d} s}\ms\bigl((u_s,\g_s);U\bigr)_{|s=t} = \int_{\g_t} f_t (X_t\cdot\nu_t)\dh + \int_{\partial{\g_t}} X_t\cdot\eta_t\dc\,.
\end{equation}
Notice that the functions $f_t$ are well defined $C^1$ functions in a normal tubular neighborhood of $\g_t$.
In particular, for $t=0$, we deduce the following expression for the first variation:
$$
\frac{\mathrm{d}}{\mathrm{d} t}{\ms\bigl((u_t,\g_t);U\bigr)}_{|t=0} = \int_{\g} \bigl(|\nabla_{\g} u^+|^2 - |\nabla_{\g} u^-|^2 + H\bigr)(X\cdot\nu)\dh + \int_{\partial{\g}} X\cdot\eta\dc\,.
$$

\emph{Computation of the second variation}. Now we want to compute
$$
\frac{\mathrm{d}^2}{\mathrm{d}s^2} {\ms\bigl((u_s,\g_s);U\bigr)}_{|s=t}\,,
$$
for $s\in(-1,1)$. The derivative of the first term of \eqref{eq:firvar} can be computed as follows:
\begin{align*}
\frac{\mathrm{d}}{\mathrm{d} t}&{\Biggl( \int_{\g_t} f_t (X_t\cdot\nu_t)\dh \Biggr)}_{|t=0} = \frac{\mathrm{d}}{\mathrm{d} t}{\Biggl( \int_{\g} (f_t\circ\Phi_t) 
				(X_t\circ\Phi_t)\cdot(\nu_t\circ\Phi_t)J_{\Phi_t}\dh \Biggr)}_{|t=0}\\
&= \int_{\g} (\dot{f}+\nabla f\cdot X)(X\cdot\nu)\dh + \int_{\g} f\frac{\partial}{\partial t}{\bigl(\dot{\Phi}_t\cdot(\nu_t\circ\Phi_t)J_{\Phi_t}\bigr)}_{|t=0} \dh \,.
\end{align*}
Now, using equality (7) of Lemma ~\ref{lem:geom} to rewrite the second integral, we get
\begin{align*}
\frac{\mathrm{d}}{\mathrm{d} t}&{\Biggl( \int_{\g_t} f_t (X_t\cdot\nu_t)\dh \Biggr)}_{|t=0} = \int_{\g} (\dot{f}+\nabla f\cdot X)(X\cdot\nu)\dh \\
&\hspace{0.3cm} + \int_{\g} f\bigl(  Z\cdot\nu-2X^{||}\cdot\ng(X\cdot\nu)+D\nu[X^{||},X^{||}]+\div_{\g}\bigl((X\cdot\nu)X\bigr) \bigr)\dh \\
&= \int_{\g} f\bigl(Z\cdot\nu-2X^{||}\cdot\ng(X\cdot\nu)+D\nu[X^{||},X^{||}] \bigr)\dh \\
&\hspace{0.3cm}+ \int_{\g} (\dot{f}+\nabla f\cdot \nu(X\cdot\nu))(X\cdot\nu) \dh + \int_{\g} \div_{\g}\bigl( f(X\cdot\nu)X \bigr)\dh \\
&= \int_{\g} (\dot{f}+\nabla f\cdot \nu(X\cdot\nu))(X\cdot\nu) \dh + \int_{\partial\g}f(X\cdot\nu)(X\cdot\eta)\mathrm{d}\mathcal{H}^0\\
&\hspace{0.3cm} +\int_{\g} Hf(X\cdot\nu)^2\dh + \int_{\g}f\bigl(Z\cdot\nu-2X^{||}\cdot\ng(X\cdot\nu)+D\nu[X^{||},X^{||}] \bigr) \dh\,,
\end{align*}
where the last equality follows from integration by parts, while the previous one by writing $X=(X\cdot\nu)\nu+X^{||}$.
Now, recalling that $f=|\nabla_{\g} u^-|^2 - |\nabla_{\g} u^+|^2 + H$, we have that
$$
\begin{array}{c}
\nabla f = 2\ng u^+\, D^2u^- - 2\ng u^-\, D^2u^+ + \nabla H\,,\\
{}\\
\dot{f}= 2\ng u^+\cdot\ng\dot{u}^- - 2\ng u^-\cdot\ng\dot{u}^+ + \dot{H}\,.
\end{array}
$$
Using the above identities and (2), (4) and (5) of Lemma ~\ref{lem:geom} we can write
\begin{align*}
\int_{\g} (\nabla f\cdot\nu)&(X\cdot\nu)^2\dh = \int_{\g} (X\cdot\nu)^2\bigl[ 2D^2 u^-[\ng u^-,\nu] - 2D^2 u^+[\ng u^+,\nu] + \partial_{\nu}H \bigr] \dh \\
&= \int_{\g} (X\cdot\nu)^2\bigl[ 2D\nu[\ng u^{+}, \ng u^{+}] - 2D\nu[\ng u^{-}, \ng u^{-}] - |D\nu|^2 \bigr]\dh \\
&= \int_{\g} (H^2-2fH)(X\cdot\nu)^2\dh\,,
\end{align*}
where the identity $D\nu[\tau,\tau]=H$ has been used in the last step.

Now we would like to treat the term $\int_{\g} \dot{f}(X\cdot\nu) \dh$. First of all we recall that $H=\div_{\g}\nu$ and $\partial_\nu\dot{\nu}=0$ (since $|\nu_t|^2\equiv1$). Thus $\dot{H}=\mathrm{div}_\g\dot{\nu}$, and hence
\begin{eqnarray*}
\int_{\g}\dot{H}(X\cdot\nu)\dh &=& \int_{\g}(\div_{\g}\dot{\nu})(X\cdot\nu)\dh \\
&=& -\int_{\g} \dot{\nu}\cdot\ng(X\cdot\nu)\dh + \int_{\partial\g}(\dot{\nu}\cdot\eta)(X\cdot\nu)\dc \\
&=& \int_{\g} |\ng (X\cdot\nu)|^2\dh + \int_{\partial\g}(\dot{\nu}\cdot\eta)(X\cdot\nu)\dc\,,
\end{eqnarray*}
where in the last line we have used $(6)$ of Lemma ~\ref{lem:geom}. Moreover
\begin{equation*}
\int_{\g} (\ng u^{\pm} \cdot \ng \dot{u}^{\pm}) (X\cdot\nu)\dh
= -\int_{\g} \dot{u}^{\pm}\div_{\g}\bigl(\ng u^{\pm}(X\cdot\nu)\bigr)\dh + 2\int_{\partial\g} \dot{u}^{\pm}(X\cdot\nu)(\ng u^{\pm}\cdot\eta)\dc\,,
\end{equation*}
Hence, recalling \eqref{eq:dotu}, we obtain
\begin{align*}
\int_{\g} \dot{f}(X\cdot\nu)\dh =& -2\int_{U}|\nabla\dot{u}|^2\dx + \int_{\g} |\ng (X\cdot\nu)|^2\dh + \int_{\partial\g}(\dot{\nu}\cdot\eta)(X\cdot\nu)\dc \\
&+ 2\int_{\partial\g} \Bigl[ \dot{u}^+(X\cdot\nu)(\ng u^+\cdot\eta) - \dot{u}^-(X\cdot\nu)(\ng u^-\cdot\eta) \Bigr] \dc \,.
\end{align*}

Finaly, we have to compute the derivative of the second integral of \eqref{eq:firstvar}. Using $(ii)$ of Lemma ~\ref{lem:geom}, we have that
\begin{eqnarray*}
\frac{\mathrm{d}}{\mathrm{d} t}{\Biggl( \int_{\partial\g_t} X_t\cdot\eta_t \dc \Biggr)}_{|t=0} &=& 
						\frac{\mathrm{d}}{\mathrm{d} t}{\Biggl( \int_{\partial\g} (X_t\circ\Phi_t)\cdot(\eta_t\circ\Phi_t) \dc \Biggr)}_{|t=0} \\
&=& \int_{\partial\g} \biggl( Z\cdot\eta + X\cdot\frac{\partial}{\partial t}{(\eta_t\circ\Phi_t)}_{|t=0} \biggr)\dc \\
&=& \int_{\partial\g} \biggl( Z\cdot\eta -(X\cdot\nu)(\dot{\nu}\cdot\eta)-H(X\cdot\nu)(X\cdot\eta)  \biggr)\dc\,.
\end{eqnarray*}

We now observe that some integrals vanishes for regular admissible triple points. Indeed, by the Neumann conditions satisfied by $u$, we know that 
$\partial_{\nu}u^{\pm}=0$ on $\g$ and that $\partial_{\nu_{\po}}u^{\pm}=0$ on $\partial_{N}\o\cap\bar{U}$. The admissibility conditions we required on regular admissible triple points tell us that $\nu_{\po}(x_i)$ and $\nu_i(x_i)$ are linear independent for every $i=1,2,3$, as well as
$\nu_1(x_0)$ and $\nu_2(x_0)$. Using the fact that $\nabla u^{\pm}$ is continuous up to the closure of $\g$, we can infer that $\nabla u^{\pm}(x_i)=0$ for each $i=0,1,2,3$.\\

Combining all the above identities, we obtain the desired formula for the second variation of our functional $\ms$ at a regular admissible triple point $(u,\g)$.
\end{proof}

\begin{remark}\label{rem:secvar}
The above expression for the second variation can be also used to compute the second variation at a generic time $t\in(-1,1)$. Indeed, fix $t\in(-1,1)$, and consider the family of diffeomorphisms
$$
\widetilde{\Phi}_{s}:=\Phi_{t+s}\circ\Phi_{t}^{-1}\,.
$$
It is easy to see that this family is admissible for $(u,\g)$ in $U$, and that
$$
\frac{\mathrm{d}^2}{\mathrm{d}s^2} {\ms\bigl((u_s,\g_s);U\bigr)}_{|s=t} =
             \frac{\mathrm{d}^2}{\mathrm{d}h^2} {\ms\bigl((u_{t+h},\widetilde{\Phi}_{h}(\g_{t}));U\bigr)}_{|h=0}\,.
$$
Hence, \emph{mutatis mutandis}, the same expression as in \eqref{eq:secvar} holds true for the second variation at a generic time $t\in(-1,1)$.
\end{remark}

The expression \eqref{eq:firstvar} of the first variation suggests the following definition.

\begin{definition}\label{def:crit}
Let $(u,\g)$ be a triple point and $U$ an admissible subdomain. We say that $(u,\g)$ is \emph{critical} if the following three conditions are satisfied:
\begin{itemize}
\item $H=|\nabla_{\g} u^-|^2 - |\nabla_{\g} u^+|^2$ on $\g$\,,
\item the $\g^i$'s meet in $x_0$ at $\frac{2}{3}\pi$\,,
\item each $\g^i$ meets $\po$ orthogonally\,.
\end{itemize}
\end{definition}

\begin{remark}
Notice that a critical triple point is such that $H_i=0$ on $\partial\g^i$.
\end{remark}

Now we want to rewrite the second variation in a critical triple point.

\begin{proposition}\label{prop:svcp}
Let $(u,\g)$ be a regular critical triple point. Then the second variation of $\ms$ at $(u,\g)$ in $U$ can be written as follows:
\begin{eqnarray*}
\frac{\mathrm{d}^2}{\mathrm{d}t^2}{\ms\bigl((u_t,\g_t); U\bigr)}_{|t=0} &=& -2\int_{U}|\nabla\dot{u}|^2\dx + \int_{\g}|\ng(X\cdot\nu)|^2\dh  \\
& & + \int_{\g}H^2(X\cdot\nu)^2\dh -\sum_{i=1}^{3}(H_{\po}(X\cdot\nu^i)^2)(x_i)\,.
\end{eqnarray*}
\end{proposition}

\begin{proof}
We notice that for a regular admissible critical triple point $f=0$ on $\g$, $\nabla u^{\pm}=0$ on $\partial\g$, $X\cdot\eta=X\cdot\nu_{\po}=0$ on
$\bar{\g}\cap\po$ and, thanks to $(iii)$ of Lemma ~\ref{lem:geom}, that
$$
Z\cdot\eta=-D\nu_{\po}[X,X]=-(X\cdot\nu)^2 D\nu_{\po}[\nu,\nu]=-H_{\po}(X\cdot\nu)^2\,.
$$
Recalling that the $\g^i$'s meet in $x_0$ at $\frac{2}{3}\pi$, we also have that $\sum_{i=1}^3 Z\cdot\nu^i(x_0)=0$.
This allows to conclude.
\end{proof}

The above result suggests to introduce the following definition.

\begin{definition}\label{def:h1}
We introduce the space
$$
\widetilde{H}^1(\g):=\{ \p:\Gamma\rightarrow\R\,:\, \p_i\in H^1(\g^i)\,,\,\bigl(\p_1+\p_2+\p_3\bigr)(x_0)=0 \}\,,
$$ 
endowed with the norm given by:
$$
\|\p\|_{\widetilde{H}^1(\g)}:=\sum_{i=1}^3\|\p_i\|_{H^1(\g^i)}\,.
$$
Then, we define the quadratic form $\partial^2\ms\bigl((u,\g);U\bigr):\widetilde{H}^1(\g)\rightarrow\R$ as
\begin{eqnarray*}
\partial^2\ms\bigl((u,\g);U\bigr)[\p] &:=& -2\int_{U}|\nabla v_{\p}|^2\dx + \int_{\g}|\ng \p |^2\dh + \int_{\g}H^2\p^2\dh \\
& &-\sum_{i=1}^{3}\bigl(\p_i^2 D\nu_{\po}[\nu,\nu]\bigr)(x_i) \,,
\end{eqnarray*}
where $v_{\p}\in H^1_U(\o\meno\g)$ is the solution of
\begin{eqnarray}\label{eq:vphi}
\int_{\o}\nabla v_{\p}\cdot\nabla z\,\mathrm{d}x &=& \langle \div_{\g}\bigl( \p\nabla_{\g}u^+ \bigr), z^+ \rangle_{H^{-\frac{1}{2}}(\g)\times H^{\frac{1}{2}}(\g)} \nonumber\\
&& \hspace{0.3cm} - \langle \div_{\g}\bigl( \p\nabla_{\g}u^- \bigr), z^- \rangle_{H^{-\frac{1}{2}}(\g)\times H^{\frac{1}{2}}(\g)}\,,
\end{eqnarray}
for every $z\in H^1_U(\o\meno\g)$.
\end{definition}

The following lemma ensures that the right-hand side of \eqref{eq:vphi} makes sense.

\begin{lemma}\label{lem:fss}
Let $\varphi\in\widetilde{H}^1(\g)$ and let $\Phi\in H^{\frac{1}{2}}(\g)\cap C^0(\g)$. Then $\varphi\Phi\in H^{\frac{1}{2}}(\g)$.
\end{lemma}

\begin{proof}
We need to estimate the Gagliardo semi-norm. So
\begin{align*}
[\varphi\Phi]^2_{H^{1/2}}&:=\int_\g\int_\g \frac{|\varphi(x)\Phi(x)-\varphi(y)\Phi(y)|^2}{|x-y|^2}\dh(x)\dh(y) \\
&\leq \int_\g\int_\g |\Phi(y)|^2\frac{|\varphi(x)-\varphi(y)|^2}{|x-y|^2}\dh(x)\dh(y)\\
&\hspace{0.5cm}+\int_\g\int_\g |\varphi(x)|^2\frac{|\Phi(x)-\Phi(y)|^2}{|x-y|^2}\dh(x)\dh(y) \\
&\leq \|\Phi\|^2_{C^0} [\varphi]^2_{H^{1/2}} + \|\varphi\|^2_{L^\infty} [\Phi]^2_{H^{1/2}}\,.
\end{align*}
Using the Sobolev embedding $H^1(\g)\subset H^{\frac{1}{2}}(\g)\cap L^\infty(\g)$, we obtain that the above quantity is finite, and hence we conclude.
\end{proof}

\begin{remark}
The above result holds just requiring $\Phi\in H^{\frac{1}{2}}(\g)$, but the proof is longer. Since in our case we already know that
$\nabla_\g u^\pm \in H^{\frac{1}{2}}(\g)\cap C^{0,\alpha}(\g)$ for $\alpha\in(0,1/2)$, we prefer to give just this simplified version of the result.
\end{remark}

\begin{remark}
Notice that it is possible to write
\begin{equation}\label{eq:varsec}
\frac{\mathrm{d}^2}{\mathrm{d}t^2} {\ms\bigl((u_t,\g_t);U\bigr)}_{|t=s}=\partial^2\ms\bigl((u_s,\g_s);U\bigr)\bigl[ (X\cdot\nu^1,X\cdot\nu^2,X\cdot\nu^3) \bigr]+R_s\,,
\end{equation}
where $R_0$ vanishes whenever $(u,\g)$ is a critical triple point.
\end{remark}

We now introduce the space where we will prove the local minimimality result.

\begin{definition}\label{def:admdiff}
Given $\delta>0$, we denote by the symbol $\diffc$ the space of all the diffeomorphisms $\Phi:\bar{\o}\rightarrow\bar{\o}$, with
$\Phi=\id$ in $(\o\meno U)\cup\pdo$, such that $\|\Phi-\id\|_{W^{2,\infty}(\g;\bar{\o})}<\delta$.
\end{definition}

Notice that we only require $W^{2,\infty}$-closeness of $\Phi$ to the identity on the set $\g$.
As one would expect, the non negativity of the second variation is a necessary condition for local minimality, as shown in the following result.
Since the proof is just technical, it will be postponed in the appendix.

\begin{proposition}\label{prop:pos}
Let $(u,\g)$ be a critical triple point such that there exists $\delta>0$ with the following property:
$$
\ms\bigl( (u,\g);U \bigr) \leq \ms\bigl( (v,\g_{\Phi});U \bigr)\,,
$$
for every diffeomorphisms $\Phi:\bar{\o}\rightarrow\bar{\o}$ with $\Phi=\id$ on $\pdo\cup (\o\meno U)$ satisfying $\|\Phi-\id\|_{C^2(\g;\bar{\o})}<\delta$, and every $v\in H^1(\o\,\meno \,\g_\Phi)$ such that $v=u$ in $(\o\meno U)\cup\pdo$.
Then
$$
\partial^2\ms\bigl((u,\g); U\bigr)[\p]\geq0\,,\quad\quad\quad\mbox{for every }\p\in\widetilde{H}^1(\g)\,.
$$
\end{proposition}

The following strict stability condition will be shown to imply the local minimality result (see Theorem~\ref{thm:minC2}).

\begin{definition}
We say that a critical triple point $(u,\g)$ is \emph{strictly stable} in an admissible subdomain $U$ if
$$
\partial^2\ms\bigl((u,\g); U\bigr)[\p]>0\quad\quad\quad\mbox{for every }\p\in\widetilde{H}^1(\g)\meno\{0\}\,.
$$
\end{definition}


\section{A local minimality result}

The aim of this section is to prove the following result.

\begin{theorem}\label{thm:minC2}
Let $(u,\g)$ be a strictly stable critical triple point. Then there exists $\bar{\delta}>0$ such that
$$
\ms\bigl((u,\g);U\bigr)\leq\ms\bigl((v,\g_{\Phi});U\bigr)\,,
$$
for every $\Phi\in\mathcal{D}_{\bar{\delta}}(\o;U)$ and every $v\in H^1(\o\meno \g_\Phi)$ such that $v=u$ in $(\o\meno U)\cup\pdo$.
Moreover equality holds true only when $\g_{\Phi}=\g$ and $v=u$.
\end{theorem}

The rest of this section is devoted to the proof of the above result.


\subsection{Construction of the admissible family}

In this section we construct a suitable admissible family connecting a critical point $\g$ with a competitor that satisfies some additional assumptions.

\begin{proposition}\label{prop:x}
Let $(u,\g)$ be a critical triple point and fix $\varepsilon>0$. Then it is possible to find a constant $\bar{\delta}_1=\bar{\delta}_1(\g, \varepsilon)>0$ and constants $C_1>0$, $C_2>0$, depending only on $\g$ and $\bar{\delta}_1$, with the following property:

for any diffeomorphism $\Phi\in C^3(\bar{\o};\bar{\o})$ such that
$\|\Phi-\id\|_{C^2(\g;\bar{\o})}<\bar{\delta}_1$ and $\Phi(x_0)\neq x_0$,
it is possible to find an admissible family $(\Phi_t)_{t\in[0,1]}$ such that
$$
\|\Phi_t-\id\|_{C^2(\g;\bar{\o})}<\varepsilon\,,\quad\quad\Phi_1(\g)=\Phi(\g)\,.
$$
Moreover the following estimates hold true for each time $t\in[0,1]$:
\begin{equation}\label{eq:stima1}
\|X_t\cdot\tau_t\|_{L^2(\g_t)}\leq C_1 \|X_t\cdot\nu_t\|_{L^2(\g_t)}\,,
\end{equation}
\begin{equation}\label{eq:stima2}
\|Z_t\cdot\nu_t\|_{L^1(\g_t)}\leq C_2\|X_t\cdot\nu_t\|_{L^2(\g_t)}\,,
\end{equation}
where we recall that $\nu_t$ and $\tau_t$ are the normal and the tangent vector field on $\g_t$ respectively and that the objects $X_t:=\dot{\Phi}_t\circ\Phi_t^{-1}$, $Z_t:=\ddot{\Phi}_t\circ\Phi_t^{-1}$ are well defined on $\g_t$.\\
\end{proposition}

\emph{Heuristics}.
Before starting with the proof, we would like to give the reader an general overview of what we are going to do.
The idea of the construction is similar to the one used in \cite{CicLeoMag}.
As explained in the introduction, since estimates \eqref{eq:stima1} and \eqref{eq:stima2} are not explicitly present in the above work, we decided, for reader's convenience, to give here another construction, where the relevant features that allow to obtain the estimates are explicitly pointed out.
Moreover, since the curves $\Phi(\g^i)$ do not necessarily meet with equal angles, we cannot use Whitney extension Theorem to extend the functions we will define on each $\g^i$ to a function of the whole $\o$. For, we use Lemma \ref{lem:ext}, that is explicitly designed for our purposes, where the extension can also failed to be a global diffeomorphism of $\o$, but with the essential properties that allow us to perform the computations of the first and the second variations (see Definition \ref{def:family}).

The source of difficulties is, of course, the presence of the triple point.
Indeed, for points in $\bar{\g}$ far for $x_0$, we can use a standard construction. Namely, we can define the family $(\Phi_t)_t$ as the flow of a vector field that is (close to) an extension of the normal vector field of $\bar{\g}$. This part of the construction is easy. 
The tricky part is when we are closed to $x_0$. The idea we are going to use is the following: we first construct a vector field $Y$ on $\bar{\g}\cap B_\mu(x_0)$, for some
$\mu>0$, such that $x+Y(x)\in\Phi(\g)$ and $Y$ has null tangential component on $\g\cap B_\mu(x_0)\meno B_{\frac{\mu}{2}}(x_0)$. This last condition will be used to glue together the vector field $Y$ with the one defined far from $x_0$. Then we define our diffeomorphisms $\Phi_t$'s on $\bar{\g}\cap B_\mu(x_0)$ as
$$
\Phi_t(x):=x+tY(x)\,.
$$
In order to obtain \eqref{eq:stima1} and \eqref{eq:stima2}, we need our vector field $Y$ to satisfy the following two conditions:
\begin{itemize}
\item[(i)] $Y\in C^3(\bar{\g}\cap B_\mu(x_0))$, with $\|Y\|_{C^2(\bar{\g}\cap B_\mu(x_0))}$ sufficiently small,
\item[(ii)] $\|Y\cdot\nu\|_{L^2(\bar{\g}\cap B_\mu(x_0))}\geq M \|Y\cdot\tau\|_{L^2(\bar{\g}\cap B_\mu(x_0))}$, for some $M>0$.
\end{itemize}
Thus, we wonder how to ensure the validity of the above conditions.
The first one is not difficult to achieve by using the assumption that $\Phi$ is closed to the identity in the $C^2$ norm.
Condition $(ii)$ is the tricky one. It suggests us to consider the sets
$$
\mathcal{C}^i:=\Bigl\{v\in B_\mu \;:\; \Bigl|\frac{v}{|v|}\cdot\tau^i(x_0)\Bigr|\geq \frac{3}{5}\Bigl|\frac{v}{|v|}\cdot\nu^i(x_0)\Bigr|\,\Bigr\}\,,\\
$$
and to use different constructions when $Y(x_0):=\Phi(x_0)-x_0\in \mathcal{C}^i$ for some $i$, or when this condition is not satisfied.
In the latter one, it is easy to define $Y$ (by letting its normal part to vanish) in such a way that
$|Y\cdot\nu|\geq C|Y\cdot\tau|$. This pointwise estimate is enough to ensure the validity of the integral estimate $(ii)$.
If instead $Y(x_0)\in \mathcal{C}^i$ for some $i$, then it is not always possible to obtain an estimate of the type
\[
\|X_t\cdot\tau_t\|_{L^2(\g^i_t\cap B_\mu(x_0))}\leq C \|X_t\cdot\nu_t\|_{L^2(\g^i_t\cap B_\mu(x_0))}\,.
\]
Just consider the the following example: each $\g^i$ is a segments and $\Phi$ is, around $x_0$, a translation in the direction of, let us say, $\g^1$. In this case, an estimate like the above one cannot be true for $\g^1$, since $Y$ has only tangential part on that curve.
The idea is to take advantage of the fact that we have a triple point.
Thus, if $Y(x_0)\in \mathcal{C}^1$, then clearly $Y(x_0)\not\in \mathcal{C}^2\cup \mathcal{C}^3$ and so, for $j=2,3$, the inequality $|Y(x_0)\cdot\tau^j(x_0)|\leq C|Y(x_0)\cdot\nu^j(x_0)|$ holds true.
This means that, for $j=2,3$, we can directly obtain the desired integral estimate from the pointwise one.
Now it is clear that the only chance we have in order to satisfy $(ii)$, is to estimate the tangential part of $Y$ along $\g^1$ with its normal part along $\g^2$, \emph{i.e.}, to obtain the following estimate
\begin{equation}\label{eq:sth}
\|Y\cdot\tau^1\|_{L^2(\g^1_t\cap B_\mu(x_0))}\leq C \|Y\cdot\nu^2\|_{L^2(\g^2_t\cap B_\mu(x_0))}\,,
\end{equation}
to hold true.

\begin{figure}[H]
\includegraphics[scale=1]{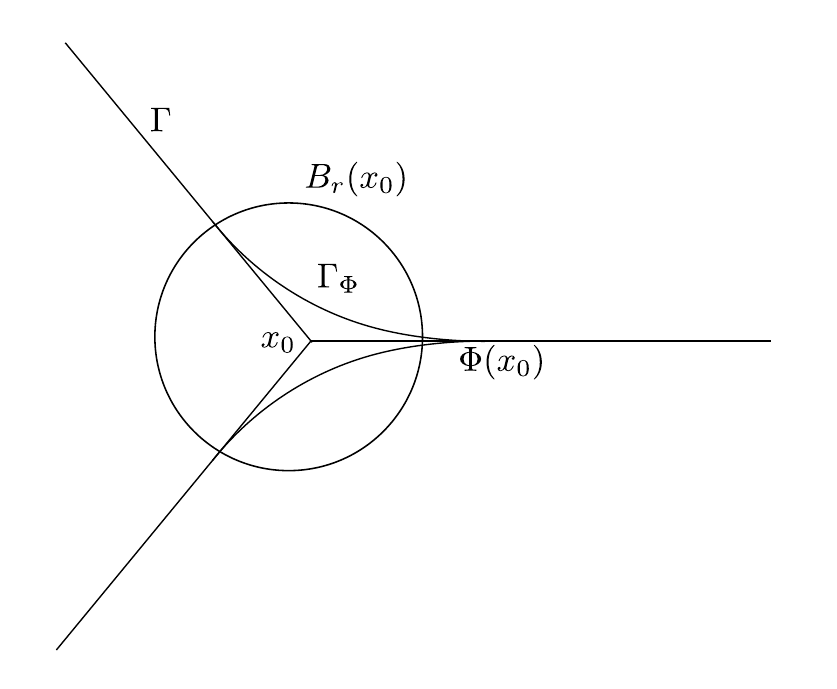}
\caption{The original triple point (bold line) and its image under the diffeomorphism (dashed line).}
\label{fig:worst}
\end{figure}

Are we sure that we can do it?
The worst case scenario is the one shown in Figure \ref{fig:worst}: the curve $\Phi(\g^1)$ is completely over $\g^1$, and $\Phi(\g^2)$ is over $\g^2$ out of a ball $B_r(x_0)$.
Using the closeness of $\Phi$ to the identity in the $C^2$ norm, a very rough estimate allows us to estimate from below $r$ with a term of the order of $\sqrt{|Y(x_0)|}$.
Thus, we need to construct our vector field $Y$ around $x_0$ in such a way that:
\begin{itemize}
\item $|Y\cdot\tau^1|\leq C|Y(x_0)|$ on $\g^1\cap B_{C\sqrt{|Y(x_0)|}}(x_0)$,
\item $Y\cdot\tau^1\equiv0$ on $\g^1\meno B_{C\mu}(x_0)$,
\item $|Y\cdot\nu^2|\geq C|Y(x_0)|$ on $\g^2$.
\end{itemize}
If the above conditions are in force, then it is easy to see that \eqref{eq:sth} holds true.
We will use two different strategies to let the tangential part of $Y$ vanish: on $\g^2$ we just use a cut off function, that will ensure the validity of the last condition.
The idea for constructing the vector field on $\g^1$ is to look at that curve and at $\Phi(\g^1)$ near $x_0$ as graphs, with respect to the axes given by $\tau^1(x_0)$ and $\nu^1(x_0)$, of two functions $h^1\in C^3([0,s])$ and $\widetilde{h}^1\in C^3([a,s])$ respectively (where $a\neq0$ since $\Phi(x_0)\neq x_0$).
We want the projection of $Y(x_0)$ on $\tau^1(x_0)$ to go to zero, and then use the fact that, for $s$ small, $Y(s,h(s))$ is almost aligned with $\nu^1(s,h(s))$. For, we notice that if $Y$ is a vector field connecting a point $(s,h(s))$ of $\mathrm{graph}(h)$ with a point $(t,h^1(t))$ of $\mathrm{graph}(h^1)$, then we are interested in making the quantity $t-s$ disappearing. Thus, we are lead to consider diffeomorphisms $G:[0,s]\rightarrow[a,s]$ and requiring that they are the identity on $[\bar{s},s]$, for some $\bar{s}\in(a,s)$.\\

\begin{proof}
The proof is divided in three parts: we first define our functions on $\bar{\g}$, then we extend them to admissible ones defined in the whole
$\bar{\o}$ and finally we will show that our construction is such that estimate \eqref{eq:stima1} and \eqref{eq:stima2} hold true. We start with some preliminaries.\\

\emph{Preliminaries}. The constant $C>0$ that will appear in the following computations may change from line to line, but we will keep the same notation. Fix $\mu>0$ such that
\begin{itemize}
\item $\nu_i(x)\cdot\nu_i(x_0)\geq\frac{2}{3}$, for $x\in \g^i\cap B_\mu(x_0)$,
\item $B_{4\mu}(x_0)\Subset\o$,
\item $(\g^i)_\mu$ is a tubular neighborhood of $\g^i$,
\item the sets $(\g^i)_\mu\meno B_{3\mu}$ are disjoint,
\item $\g^i\cap B_\mu(x_0)$ is a graph with respect to the axes given by $\tau^i(x_0)$ and $\nu^i(x_0)$.
\end{itemize}
We will take $\bar{\delta}_1<\frac{\mu}{2}$.
Moreover, we will fix a cut-off function $\chi:\R\rightarrow[0,1]$ such that $\chi\equiv0$ on $[1,+\infty)$ and $\chi\equiv1$ on
$(-\infty,\frac{1}{2}]$.\\


\emph{Step 1: construction of the functions near $x_0$}. We define the functions $\Phi^O_t$ as
$$
\Phi^O_t(x):=x+tN(x)\,,
$$
for $x\in\bar{\g}\cap B_\mu(x_0)$, where the vector field $N$ will be constructed as follows.\\

\emph{Case 1: $x_0\not\in C^i$}. Write
$$
\frac{\Phi(x)-x}{|\Phi(x)-x|} = a_i(x)\tau^i(x) + b_i(x)\nu^i(x)\,,
$$
for some functions $a_i, b_i:\bar{\g}^i\cap B_\mu(x_0)\rightarrow\R$.
Notice that $\Phi(x_0)\neq x_0$ allows us to say that $\Phi(x)\neq x$ in a neighborhood of $x_0$.
Up to take a smaller $\mu$, we can suppose $|b_i(x)-b_i(x_0)|<\frac{1}{4}$ and $|a_i(x)-a_i(x_0)|<\frac{1}{4}$ for $x\in\g^i\cap B_\mu(x_0)$.
Notice that, since $x_0\not\in C^i$, we have $|b_i(x_0)|\geq \frac{3}{4}$, $|a_i(x_0)|\leq \frac{\sqrt{5}}{4}$.
Consider the unitary vector field $Y^i$ on $\bar{\g}^i\cap B_{3\mu}(x_0)$ given by
$$
Y^i(x):=\frac{\widetilde{Y}^i}{|\widetilde{Y}^i|}\,.
$$
where, if we define $\widetilde{\chi}(x):=\chi\Bigl(\frac{|x-x_0|^2}{\mu^2} \Bigr)$, we set
$$
\widetilde{Y}^i:=\widetilde{\chi}(x)\bigl( a_i(x)\tau^i(x) + b_i(x)\nu^i(x) \bigr) + ( 1- \widetilde{\chi}(x)) \nu^i(x)\,.
$$
Then there exists a constant $C>0$, depending only on $\g$ and $\chi$, such that
\begin{equation}\label{eq:ue}
|Y^i\cdot\nu^i|\geq C |Y^i\cdot\tau^i|\,,
\end{equation}
on $\g^i\cap B_{3\mu}(x_0)$. Indeed, by choosing $\bar{\delta}_1$ small enough, we have that
\begin{eqnarray*}
|\widetilde{\chi}(x) a_i(x) | &\leq & |\widetilde{\chi}(x)(a_i(x)-a_i(x_0))|+|\widetilde{\chi}(x)a_i(x_0)|\leq \frac{1}{4}+|a_i(x_0)|\\
&\leq& C \leq 1-|b_i(x)-1| \leq |1+\widetilde{\chi}(x)(b_i(x)-1)|\,,
\end{eqnarray*}
where in the second to last inequality we have used the fact that $|b_i(x)-1|\leq \frac{1}{2}$.
Moreover, it is possible to find a constant $C>0$ independent of $a_i(x_0)$, $b_i(x_0)$, such that $\|Y^i\|_{C^3(\bar{\g}^i\cap B_{3\mu}(x_0))}\leq C$.
We claim that it is possible to represent (a piece of)
$\Phi(\g^i)$ as a graph of class $C^3$ over $\g^i$, with respect to the vector field $Y^i$.
Namely, it is possible\footnote{This is an application of the implicit function theorem, by using the uniform estimate \eqref{eq:ue}} to find a function $\varphi^i\in C^3(\g^i\cap B_{3\mu}(x_0))$ such that
$$
x\mapsto x+\varphi^i(x)Y^i(x)
$$
is a diffeomorphisms of class $C^3$ from $\g^i\cap B_{3\mu}(x_0)$ to its image, that is contained in $\Phi(\g^i)$.
Finally, for any $\xi>0$ it is possible to find $\bar{\delta}_1>0$ such that if $\|\Phi-\id\|_{C^2(\g;\bar{\o})}<\bar{\delta}_1$, then
$\|\varphi\|_{C^3(\g^i\cap B_{3\mu}(x_0))}<\xi$. Define $N:=\varphi Y^i$.\\


\emph{Case 2: $x_0\in C^i$}. Consider the axes given by $\tau^i(x_0)$ and $\nu^i(x_0)$ centered at $x_0$, and denote by $s$ the coordinate with respect to $\tau^i(x_0)$.
Our assumptions on $\mu$ allow us to write $\g^i$ in a neighborhood of $x_0$ as a graph of a function $h_i$, with respect to the above axes.
We can suppose $\bar{\delta}_1>0$ so small such that the same is true also for $\Phi(\g^i)$, \emph{i.e.}, we can represent $\Phi(\g^i)$ in a
neigborhood of $\Phi(x_0)$ as the graph of a function $\widetilde{h}_i$ with respect to the same axes (see Figure ~\ref{fig:axes}).

\begin{figure}[H]
\includegraphics[scale=1.2]{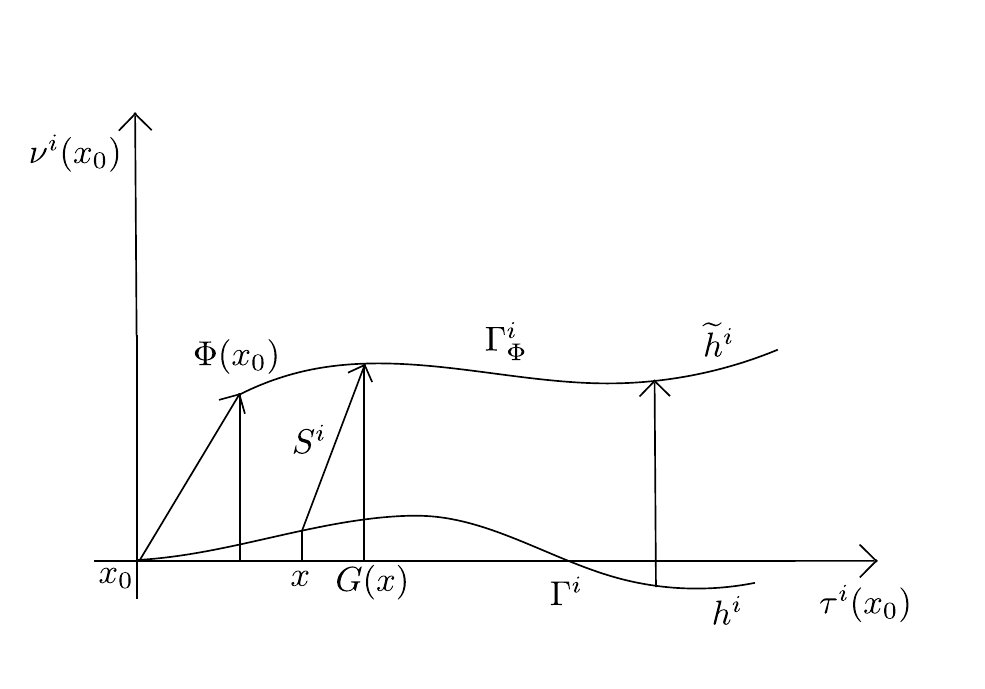}
\caption{The idea of the construction of the vector field $S^i$.}
\label{fig:axes}
\end{figure}

Now write $\Phi(x_0)-x_0 = s_0 \tau^i(x_0) + t_0 \nu^i(x_0)$, for some $s_0, t_0\in\R$, where we can also suppose $s_0<1$, if $\bar{\delta}_1$ is sufficiently small.
Since $x_0\in C^i$, we have that $C_1 s_0\leq |\Phi(x_0)-x_0|\leq C_2 s_0$, for some $C_1, C_2>0$.
For $L>1$ define the diffeomorphism $G_L:[0,(L+1)\sqrt{s_0}]\rightarrow[s_0, (L+1)\sqrt{s_0}]$ by
$$
G_L(s):=s+\chi\Bigl( \frac{s}{L\sqrt{s_0}} \Bigr)s_0\,.
$$
Notice that
\begin{equation}\label{eq:vic}
|G'_L(s)-1|\leq C\frac{\sqrt{s_0}}{L}\,,\quad\quad\quad\quad |G''_L(s)|\leq \frac{C}{L^2}\,.
\end{equation}
Moreover $G_L$ is the identity in $[L\sqrt{s_0}, (L+1)\sqrt{s_0}]$.
Now define the vector field
$$
S^i\bigl( (s,h^i(s)) \bigr):= \bigl( G_L(s)-s, \widetilde{h}_i(G_L(s))-h_i(s) \bigr)\,.
$$
Then, by a direct computation, we have
\begin{align*}
\|S^i\|_{C^2}&\leq \Bigl[ \Bigl( C\|h'\|_{C^0} + L^2(\|h'\|_{C^0} + \|\widetilde{h}'\|_{C^0})  \Bigr)s_0 \nonumber \\
&\hspace{0.3cm}+ \Bigl( \frac{C}{L} + \Bigl( 1+\frac{C}{L} \Bigr)\|\widetilde{h}_i'\|_{C^0} + \|h_i'\|_{C^0} \Bigr) \nonumber \\
&\hspace{0.3cm}+ \Bigl( \frac{C}{L^2} + \Bigl( 1+\frac{C}{L} \Bigr)^2\|\widetilde{h}_i''\|_{C^0} + \frac{C}{L^2}\|\widetilde{h}_i'\|_{C^0} + \|h_i''\|_{C^0} \Bigr)\Bigr]\,.
\end{align*}
Notice that $S^i\bigl((s,h(s))\bigr)=\lambda(s)\nu^i(x_0)$, for $\lambda(s)\in\R$ and $s\in[L\sqrt{s_0}, (L+1)\sqrt{s_0}]$.
We now want to extend the definition of the vector field $S^i$ to the whole $\bar{\g}^i\cap B_\mu(x_0)$.
Write
$$
\nu^i(x_0) = a_i(x)\tau^i(x) + b_i(x)\nu^i(x)\,,
$$
for some functions $a_i, b_i:\bar{\g}^i\cap B_\mu(x_0)\rightarrow\R$. Let $\bar{x}\in\g^i$ the point given by
\[
\bigl(\,(L+1)\sqrt{s_0}\,,\, h\bigl(\,(L+1)\sqrt{s_0}\,\bigr)\,\bigr)\,,
\]
and let $r>0$ be such that the ball $B_r(x_0)$ intersects the curve $\g^i$ in the point $\bar{x}$.
Up to take a smaller $\mu$, we can suppose $|b_i(x)-b_i(\bar{x})|<\frac{1}{4}$ and $|a_i(x)-a_i(\bar{x})|<\frac{1}{4}$, for $x\in\g^i\cap B_\mu(x_0)$.
Up to decreasing the value of $\bar{\delta_1}$, we can also suppose $|b_i(\bar{x})|\geq \frac{3}{4}$, $|a_i(\bar{x})|\leq \frac{1}{4}$.
By reasoning as in the previous step, let us consider the unit vector field 
$$
Y^i(x):=\frac{\widetilde{Y}^i}{|\widetilde{Y}^i|}\,,
$$
where we define the vector field $\widetilde{Y}$ as
$$
\widetilde{Y}^i:=\widetilde{\chi}(x)\bigl( a_i(x)\tau^i(x) + b_i(x)\nu^i(x) \bigr) + ( 1- \widetilde{\chi}(x)) \nu^i(x)\,.
$$
Using the same computation of the previous step, we have that $|Y^i\cdot\nu^i|\geq C |Y^i\cdot\tau^i|$ on $\g^i\cap B_{3\mu}(x_0)\meno B_r(x_0)$,
for some constant $C>0$. Moreover, it is possible to represent (a piece of) $\Phi(\g^i)$ 
as a graph of a function $\varphi$ of class $C^3$ over the curve $\g^i\cap B_{3\mu}(x_0)\meno B_r(x_0)$, with respect to the vector field $Y^i$.
Notice that the vector field $Y^i$ turns out to be of clas $C^3$.
Finally, for any $\xi>0$ it is possible to find $\bar{\delta}_1>0$ such that if $\|\Phi-\id\|_{C^2(\g;\bar{\o})}<\bar{\delta}_1$, then
$\|\varphi\|_{C^3(\bar{\g}^i\cap B_{3\mu}(x_0)\meno B_r(x_0))}<\xi$. Define
$$
N:=
\left\{
\begin{array}{ll}
S^i & \text{ on } \bar{\g}^i\cap B_r(x_0)\\
&\\
\varphi Y^i & \text{ on } \bar{\g}^i\cap B_{3\mu}(x_0)\meno B_r(x_0)\\
\end{array}
\right. 
$$
Notice that $N$ turns out to be a well defined $C^3$ vector field in $\g\cap B_{3\mu}(x_0)$.\\


\emph{Step 2: construction of the functions far from $x_0$}. Let $R\in C^3(\bar{\o};\R^2)$ be a vector field with the following properties
\begin{itemize}
\item $|R|\leq1$,
\item $R\bigl(x+t\nu^i(x)\bigr)=\nu_i(x)$ for any $|t|<\mu$ and any $x\in(\g^i)_\mu\meno \bigl(B_\mu(x_0)\cup(\po)_\mu\bigr)$,
\item $|R\cdot\nu^i|\geq\frac{1}{2}$ on $\g^i$,
\item $R$ is tangential to $\po$,
\item $R\equiv0$ on $\pdo\cup (\o\meno U)$.
\end{itemize}
Then, it is possible to find a function $\psi\in C^3((\g^i)_\mu\meno (B_\mu(x_0)\cup(\po)_\mu)$ (extended in a constant way along the trajectories of $R$) such that, if we consider the flow $\Phi^B_t$ of the vector field $\psi R$, we have $\Phi^B_t((\g^i)_\mu\meno (B_\mu(x_0)\cup(\po)_\mu)\in \Phi(\g^i)$.
Moreover, for any $\xi>0$ there exists $\bar{\delta}_1>0$ such that if $\|\Phi-\id\|_{C^2(\bar{\o};\bar{\o})}<\bar{\delta}_1$, then $\|\psi\|_{C^2}<\xi$.\\


\emph{Step 3: definition of the functions in $\bar{\g}$}. We define our family of functions $(\Phi_t)_{t\in[0,1]}$ as follows:
$$
\Phi_t(x):=\chi\Biggl( \frac{|x-x_0|^2}{(3\mu)^2} \Biggr)\Phi_t^O(x) + \Biggl(1-\chi\Biggl( \frac{|x-x_0|^2}{(3\mu)^2} \Biggr)\Biggr)\Phi^B_t(x)\,.
$$
Notice that the flows $\Phi^O_t$ and $\Phi^B_t$ are the same for points $x\in \g\meno(B_{2\mu}(x_0)\cup(\po)_\mu)$.
Moreover\footnote{Recall that $x_0\not\in\g$.} the above functions are of class $C^3$ in $\g$ and $\Phi_1(\bar{\g})=\g_\Phi$.
Notice also that, for any $x\in\bar{\g}$, the function $t\mapsto\Phi_t(x)$ is of class $C^3$. \\

We claim that it is possible to find $\bar{\delta}_1>0$ and $L>1$ (where $L$ is the constant used in the construction of the functions $G_L$ in the second case of the first step) such that if $\|\Phi-\id\|_{C^2(\g;\bar{\o})}<\bar{\delta}_1$, then (up to take a smaller $\mu$)
\begin{equation}\label{eq:stnorma}
\|\Phi_t-\id\|_{C^2(\g;\bar{\o})}\leq \varepsilon\,.
\end{equation}
Indeed, we first choose $L>1$ such that $\frac{C}{L}<\frac{\varepsilon}{4}$ (where $C$ is the constant appearing in \eqref{eq:vic}), and 
then we choose $\bar{\delta}_1$ in such a way that the desired estimate holds true.\\


\emph{Step 4: extension of the functions to the whole $\bar{\o}$}. First of all we extend our functions $\Phi_t$ on $\po$ in the following way: let $x\in\po$ and consider the image of the point $x$ at time $t$ through the flow given by the vector field $\psi R$ (where $\psi$ is the function found in Step 2). Since $R$ is tangential to $\po$, we know that
if we start from a point $x\in\po$, its evolution with respect to the above flow will belong to $\po$.
Moreover notice that this functions turns out to be the identity on $\pdo$.
In order to extend each $\Phi_t$ to the whole $\bar{\o}$ we will make use of Lemma \ref{lem:ext} on each of the three connected components of $\o\meno\bar{\g}$, where the function $g$ of the lemma can be taken as the identity map in
$\Omega\meno\bigl((\g)_\delta\cup(\po)_\delta\bigr)$. We claim that $\Phi_t$ is a diffeomorphism in each connected components of $\o\meno\bar{\g}$. This can be easily seen by noticing that the extension provided by Lemma \ref{lem:ext} is close to the identity if the original functions on the curves are. Moreover it is also easy to see that we have 
$C^3$ regularity for the map $t\mapsto\Phi_t(x)$, for any $x\in\bar{\o}$. Hence, $(\Phi_t)_t$ turns out to be an admissible family.\\


\emph{Step 5: estimates}. First of all we prove estimate \eqref{eq:stima1}. By definition we have
$$
Z(x)= \Biggl(1-\chi\Biggl( \frac{|x-x_0|^2}{(3\mu)^2} \Biggr)\Biggr) Z^B(x)\,,
$$
where $Z^B(x):=\psi^2 DR[R]$ (recall that $\psi$ is the function given by Step 2). Since $|R\cdot\nu|\geq \frac{1}{2}$ in the region where we consider the flow of the vector field $\psi R$, we can take $\bar{\delta}_1$ so small such that $|R(\Phi^B_t(x))\cdot \nu_t(\Phi^B_t(x))|\geq\frac{1}{4}$ for
$x\in\g\meno B_\mu(x_0)$. Thus
$$
\int_{\g_t} |Z\cdot\nu_t|\dh = \int_{\g_t} \psi^2 DR[R,\nu_t]\dh \leq C\int_{\g_t} \psi^2|R\cdot\nu_t|^2\dh = C \int_{\g_t} |X\cdot\nu_t|^2\dh\,.
$$

To prove estimate \eqref{eq:stima2} we first need to notice the following fact: let $\alpha_1, \alpha_2>0$ be small parameters, and take
$a, b\in\R$ small such that $b\neq0$ and $|b|\geq C|a|$ for some constant $C>0$. Consider the two parabolas given by
$$
y= -\alpha_2(x-a)^2+b\,,\quad\quad, \text{ and } y=\alpha_1x^2\,.
$$
Then the distance between these two parabola are greater than $\frac{1}{2}b$ if $x\in[0,C\sqrt{b}]$ (or $x\in[a,C\sqrt{b}]$ if $a<0$), for some constant
$C>0$ depending on $\alpha_1$ and $\alpha_2$.

We use the above observation in this way: suppose $x_0\not\in C^i$ and represent the curves $\g^i$ and $\Phi(\g^i)$ in a neighborhood of $x_0$
as the graphs, with respect to the axes given by $\tau^i(x_0)$ and $\nu^i(x_0)$ centered at $x_0$, of $h_i$ and $\widetilde{h}_i$ respectively.
Up to change $\nu^i(x_0)$ with $-\nu^i(x_0)$ we can suppose $\widetilde{h}_i\geq0$. Thus, it is possible to find $\alpha_1, \alpha_2>0$ such that
$$
h_i(s)\leq \alpha_1 s^2\,,\quad\quad\quad \widetilde{h}_i(s)\geq -\alpha_2(s-a)^2+b\,,
$$
where we write $\Phi(x_0)-x_0=a\tau^i(x_0)+b\nu^i(x_0)$, for some $a, b\in\R$ with $b\neq0$ and $|b|\geq C|a|$.
Set $d:=|\Phi(x_0)-x_0|$. Thanks to the above observation we can say that
$$
|Y^i(x)|\geq\frac{1}{2}d\,,
$$
for $x\in\g^i\cap  B_{D\sqrt{d}}(x_0)$, for some constant $D>0$ depending on $\g^i$ and $\bar{\delta}_1$.

We are now in a position to prove estimate \eqref{eq:stima2}. Suppose $x_0\not\in \cup_{i=1}^3 C^i$. Thanks to the definition of the vector field
$N$ and the properties of $R$, we know that on $\g$ it holds
\begin{equation}\label{eq:1}
|X\cdot\nu|\geq C|X\cdot\tau|\,,
\end{equation}
for some constant $C>0$. Thus, a similar inequality holds on $\g_t$ provided $\bar{\delta}_1$ is sufficiently small. Hence the integral estimate follows directly.

If instead $x_0\in C^i$ we have, for $j\neq i$, the following estimate in force
\begin{equation}\label{eq:2}
\int_{\g^i\cap B_{r\sqrt{d}}(x_0)} |X\cdot\tau^i|^2\dh \leq Cd^{\frac{3}{2}}\leq C\int_{\g^j\cap B_{D\sqrt{d}}(x_0)} |X\cdot\nu_j|^2\dh\,.
\end{equation}
For $\bar{\delta}_1$ sufficiently small, the same estimate continues to hold also for the curves $\g^i_t$ and $\g^j_t$ (with $\tau^i_t$ and $\nu^i_t$).
Notice that in $\g^i\cap B_{3\mu}(x_0)\meno B_r(x_0)$ we have estimate \eqref{eq:1} in force.
By using \eqref{eq:1} and \eqref{eq:2} we obtain estimate \eqref{eq:stima2}.
\end{proof}

\begin{remark}\label{rem:cont}
From the above proof, it is easy to see that the following property holds: if $(\Phi_\varepsilon)_\varepsilon$ is a family of diffeomorphisms of class $C^2$ with the same properties as in the statement of the theorem, such that $\Phi_\varepsilon\rightarrow\Phi$ in the
$C^1$ topology, where $\Phi$ is a diffeomorphism satisfying $\Phi(\g)\neq\g$, then there exists a constant $C>0$ such that
$$
\|X^\varepsilon\cdot \nu_t^\varepsilon\|_{L^2(\g^\varepsilon_t)}\geq C\,,
$$
where $\g^\varepsilon_t:=\Phi_\varepsilon^t(\g)$, denoting by $\Phi_\varepsilon^t$ the flow generated by $X^\varepsilon$.
\end{remark}


\subsection{Uniform coercivity of the quadratic form}

The second technical result we prove is a sort of continuity of the quadratic form $\partial^2\ms\bigl((u,\g);U\bigr)$ in a stable critical triple point $(u,\g)$. This result is the fundamental estimate needed in order to prove Theorem ~\ref{thm:minC2}.

\begin{proposition}\label{prop:cont}
Let $(u,\g)$ be a strictly stable critical triple point. Then there exist $\bar{\delta}_2>0$ and $\bar{C}>0$ such that
$$
\partial^2\ms\bigl((u_{\Phi},\g_{\Phi}); U\bigr)[\p] \geq \bar{C}\|\p\|^2_{\widetilde{H}^1(\g_{\Phi})} \,,
$$
for each $\Phi\in\diffc$, where $\delta\in(0,\bar{\delta}_2)$, and each $\p\in\widetilde{H}^1(\g_{\Phi})$.
\end{proposition}

In order to prove the above proposition, we first need to prove that if $(u,\g)$ is strictly stable, then $\partial^2\ms\big( (u,\g);U \big)$ is coercive.

\begin{lemma}\label{lem:coerc}
Let $(u,\g)$ be a strictly stable critical triple point. Then there exists $M>0$ such that
$$
\partial^2\ms\bigl((u,\g);U\bigr)[\p] \geq M\|\p\|^2_{\widetilde{H}^1(\g)} \,,\quad\quad\quad\forall\p\in\widetilde{H}^1(\g)\,.
$$
\end{lemma}

\begin{proof}
It is sufficient to show that
$$
M:=\inf\{ \partial^2\ms\bigl((u,\g);U\bigr)[\p] \,:\, \|\p\|_{\widetilde{H}^1(\g)}=1 \}>0\,.
$$
Suppose for the sake of contradiction that $M=0$, and let $(\p_n)_n$ be a minimizing sequence for $M$, \emph{i.e.}, $\|\p_n\|_{\widetilde{H}^1(\g)}=1$ and
$\partial^2\ms\bigl((u,\g);U\bigr)[\p_n]\rightarrow0$. Then there exists $\p\in\widetilde{H}^1(\g)$ such that, up to a not relabelled subsequence, $\p^i_n\rightharpoonup\p^i$ in $\widetilde{H}^1(\g)$ and, by the Sobolev embeddings, $\p_n\rightarrow\p$ in $C^{0,\beta}(\bar{\g})$ for each $\beta\in(0,\frac{1}{2})$, and $\p_n\rightarrow\p$ in $H^{\frac{1}{2}}(\g)$. We claim that
\begin{equation}\label{eq:conv}
\partial^2\ms\bigl((u,\g);U\bigr)[\p] \leq\liminf_{n\rightarrow\infty}\partial^2\ms\bigl((u,\g);U\bigr)[\p_n]= 0\,.
\end{equation}
Indeed, it is easy to see that
$$
\int_{\g}|\ng \p |^2\dh \leq \liminf_{n\rightarrow\infty} \int_{\g}|\ng \p_n |^2\dh\,,
$$
$$
\int_{\g}H^2\p_n^2\dh \rightarrow \int_{\g}H^2\p^2\dh\,,
$$
and
$$
\sum_{i=1}^{3}\bigl(\p_i^2 D\nu_{\po}[\nu,\nu]\bigr)(x_i) \rightarrow \sum_{i=1}^{3}\bigl(\p_i^2 D\nu_{\po}[\nu,\nu]\bigr)(x_i) \,.
$$
Thus, we are left to prove that
$$
\int_{\g} z^{\pm}\div_{\g}(\p_n\ng u^{\pm}) \dh \rightarrow \int_{\g} z^{\pm}\div_{\g}(\p\ng u^{\pm}) \dh\,,
$$
for all $z\in H^1_U(\o\meno\g)$. Notice that $\p_n\ng u^{\pm}\in H^{\frac{1}{2}}(\g;\R^2)$ thanks to Lemma ~\ref{lem:fss}. To prove the above convergence we will show that $\p_n\ng u^{\pm}\rightarrow\p\ng u^{\pm}$ in $H^{\frac{1}{2}}(\g;\R^2)$:
\begin{align*}
\int_{\g}\int_{\g}&\frac{|(\p_n\ng u^{\pm}-\p\ng u^{\pm})(x) - (\p_n\ng u^{\pm}-\p\ng u^{\pm})(y)|}{|x-y|^2}\,\dh(x)\,\dh(y) \\
&\leq \|\ng u^{\pm}\|^2_{L^{\infty}(\bar{\g};\R^2)}\|\p_n-\p\|^2_{H^{\frac{1}{2}}(\g)} + \|\p_n-\p\|^2_{L^{\infty}(\bar{\g})}\|\ng u^{\pm}\|^2_{H^{\frac{1}{2}}(\g;\R^2)}\,.
\end{align*}
Now we have two cases: if $\p\neq0$ then \eqref{eq:conv} gives the desired contradiction. On the other hand, if $\p=0$, then $v_{\p}=0$, and
hence again by \eqref{eq:conv} we obtain that
$$
 \int_{\g}|\ng \p_n |^2\dh\rightarrow0\,,
$$
and this contradicts the fact that $\|\p_n\|_{\widetilde{H}^1(\g)}=1$.
\end{proof}

Before proving Proposition ~\ref{prop:cont} we need to observe the following fact, similar to \cite[Lemma 5.1]{BonMor}.

\begin{remark}\label{rem:conv}
Consider the function $u_{\Phi}$ (see Definition ~\ref{def:umod}). We claim that, for every $\alpha<\frac{1}{2}$, the following convergence holds true:
$$
\sup_{\Phi\in\diffc}\|\nabla_\g (u_\Phi^\pm\circ\Phi ) - \nabla_\g u^\pm \|_{C^{0,\alpha}(\bar{\g};\R^2)}\rightarrow0\,,
$$
as $\delta\rightarrow0^+$. First of all we notice that, what we are really claiming is that, denoting by $\o_1, \o_2, \o_3$ the three connected components of $\o\meno\bar{\g}$, and letting $u^i$ be the function $u$ restricted to $\o_i$, we have that
$$
\sup_{\Phi\in\diffc}\|\nabla_\g (\widetilde{u}^i_\Phi\circ\Phi ) - \nabla_\g \widetilde{u}^i \|_{C^{0,\alpha}(\bar{\g}\cap\partial \o_i;\R^2)}\rightarrow0\,,
$$
as $\delta\rightarrow0^+$, where $\widetilde{u}^i$ is the trace of $u^i$ on $\bar{\g}\cap\partial \o_i$.
This can be proved by using the estimate of the $H^2$-norm of $\widetilde{u}^i_\Phi\circ\Phi$ in a neighborhood of $\g$ (that turns out to be uniform for $\Phi\in\diffc$) and by the Ascoli-Arzel\`{a} theorem.
\end{remark}

\begin{proof}[Proof of Proposition ~\ref{prop:cont}]
Suppose for the sake of contradiction that there exist a family of diffeomorphisms $\Phi_n:\bar{\o}\rightarrow\bar{\o}$ with $\Phi_n=\id$ in $(\o\meno U)\cup\pdo$ such that $\Phi_n\rightarrow \id$ in $C^2(\bar{\o};\bar{\o})$, and functions $\p_n\in\widetilde{H}^1(\g_{\Phi_n})$ with $\|\p_n\|_{\widetilde{H}^1(\g_{\Phi_n})}=1$, such that
\begin{equation}\label{eq:abs}
\partial^2\ms\bigl( (u_{\Phi_n},\g_{\Phi_n}); U  \bigr)[\p_n]\rightarrow0\,.
\end{equation}
Let $\widetilde{\p}_n:=c_n\p_n\circ\Phi_n$, where $c_n:=\|  \p_n\circ\Phi_n \|^{-1}_{\widetilde{H}^1(\g)}\rightarrow1$. Then it is not difficult to prove that
$$
\biggl| \int_{\g_{\Phi_n}} H^2_{\Phi_n}\p^2_n \dh - \int_{\g} H^2\widetilde{\p}_n^2 \dh \biggr|\rightarrow0\,,
$$
$$
\biggl| \int_{\g_{\Phi_n}} |\nabla_{\g_{\Phi_n}}\p_n|^2 \dh - \int_{\g} |\ng\widetilde{\p}_n|^2 \dh \biggr| \rightarrow0\,,
$$
and
$$
\Bigl|\sum_{i=1}^{3}\bigl((\p_n)_i^2 D\nu_{\po}[\nu,\nu]\bigr)(\Phi_n(x_i)) - \sum_{i=1}^{3}\bigl((\widetilde{\p}_n)_i^2 D\nu_{\po}[\nu,\nu]\bigr)(x_i)\Bigr| \rightarrow 0 \,.
$$
We also claim that the following convergence holds
\begin{equation}\label{eq:convmeas}
\int_U |\nabla v_{\widetilde{\p}_n}-\nabla(v_{\p_n}\circ\Phi_n)|^2 \dx \rightarrow0\,.
\end{equation}
To prove it we proceed as in the proof of \cite[Lemma 5.4]{CagMorMor} that, for the reader's convenience, we reproduce here. Our argument only changes from the original one in the proof of the last convergence, where we take advantage of the fact that in dimension $2$ functions in $H^1(\g)$ are bounded in $L^{\infty}$. Otherwise we would have needed that $\ng(u^{\pm}\circ\Phi_n)\rightarrow\ng u^{\pm}$ in $C^{0,\alpha}(\bar{\g};\R^2)$ for some $\alpha>\frac{1}{2}$, while in our case, due to the singularity given by the triple point, we only have the above convergence for $\alpha<\frac{1}{2}$.
So, setting $z_n:=v_{\widetilde{\p}_n}-v_{\p_n}\circ\Phi_n$, we obtain that $z_n$ solves the problem
$$
\int_U A_n[\nabla z_n,\nabla z]\dx - \int_U (A_n-\id)[\nabla \widetilde{\p}_n,\nabla z]\dx + \int_{\g} (h_n^+z^+-h_n^-z^-)\dh=0\,,
$$
for all $z\in H^1_U(\o\meno\g)$, where $h_n^{\pm}:=\div_{\g}(\widetilde{\p}_n\ng u^{\pm})-\bigl( \div_{\g_{\Phi_n}}(\p_n\nabla_{\g_{\Phi_n}}u^{\pm}_{\Phi_n}) \bigr)J_{\Phi_n}$ and $A_n:=(J_{\Phi_n}D^{-1}\Phi_n D^{-T}\Phi_n)\circ\Phi_n$. Since $A_n\rightarrow\id$ in $C^1$ and the sequence $(v_{\widetilde{\p}_n})_n$ is bounded in $H^1(\o\meno\g)$, we have that $(A_n-\id)[\nabla \widetilde{\p}_n]\rightarrow0$ in $H^1(\o\meno\g)$. Thus 
\eqref{eq:convmeas} follows by showing that $h_n^{\pm}\rightarrow0$ in $H^{-\frac{1}{2}}(\g)$.
First of all we want to write the last term of $h_n$ in a divergence form. For let $\xi\in C^{\infty}_{c}(\g)$ and write
\begin{align*}
\int_{\g} & \bigl( \div_{\g_{\Phi_n}}(\p_n\nabla_{\g_{\Phi_n}}u^{\pm}_{\Phi_n}) \bigr)J_{\Phi_n}\xi\dh =
          \int_{\g_{\Phi_n}} \div_{\g_{\Phi_n}}(\p_n\nabla_{\g_{\Phi_n}}u^{\pm}_{\Phi_n})(\xi\circ\Phi_n^{-1})\dh \\
&= -\int_{\g_{\Phi_n}} \p_n\nabla_{\g_{\Phi_n}}u^{\pm}_{\Phi_n} \nabla_{\g_{\Phi_n}}(\xi\circ\Phi_n^{-1})\dh \\
&= -\int_{\g_{\Phi_n}} \p_n(D_{\g}\Phi_n)^{-T}\circ\Phi_n^{-1}[\ng(u_{\Phi_n}^{\pm}\circ\Phi_n)\circ\Phi_n^{-1}]\cdot(D_{\g_{\Phi_n}}\Phi_n)^{-T}
          [(\ng\xi)\circ\Phi_n^{-1}]\dh \\
&= -\int_{\g}c_n^{-1}\widetilde{\p}_n(D_{\g}\Phi_n)^{-1}(D_{\g}\Phi_n)^{-T}[\ng(u_{\Phi_n}^{\pm}\circ\Phi_n),\ng\xi]J_{\Phi_n}\xi\dh\\  
&= \int_{\g}c_n^{-1}\div_{\g}\bigl(\widetilde{\p}_n(D_{\g}\Phi_n)^{-1}(D_{\g}\Phi_n)^{-T})[\ng(u_{\Phi_n}^{\pm}\circ\Phi_n)]J_{\Phi_n}\bigr)\xi\dh\\        
\end{align*}
Thus we have that
$$
h_n^{\pm} = \div_{\g}\bigl( \widetilde{\p}_n\ng u^{\pm} - c_n^{-1}\widetilde{\p}_n(D_{\g}\Phi_n)^{-1}(D_{\g}\Phi_n)^{-T}[\ng(u_{\Phi_n}^{\pm}\circ\Phi_n)]J_{\Phi_n} \bigr)=:\div_{\g}\Psi_n^{\pm}\,,
$$
and hence, in order to prove that $h_n^{\pm}\rightarrow0$ in $H^{-\frac{1}{2}}(\g)$ we will prove that $\Psi_n^{\pm}\rightarrow0$ in $H^{\frac{1}{2}}(\g)$.
In order to estimate the Gagliardo $H^{\frac{1}{2}}$-seminorm, we first simplify our notation by setting $\lambda_n:=c_n^{-1}(D_{\g}\Phi_n)^{-1}(D_{\g}\Phi_n)^{-T}J_{\Phi_n}$ and $u_n:=u_{\Phi_n}^{\pm}\circ\Phi_n$. Then we can proceed as follows:
\begin{align*}
(\widetilde{\p}_n\lambda_n\ng u_n^{\pm})(x) &- (\widetilde{\p}_n\ng u^{\pm})(x) - (\widetilde{\p}_n\lambda_n\ng u_n^{\pm})(y) + (\widetilde{\p}_n\ng   
                 u^{\pm})(y)\\
&= [(\widetilde{\p}_n(\lambda_n-\id)\ng u_n^{\pm})(x) - (\widetilde{\p}_n(\lambda_n-\id)\ng u_n^{\pm})(x) ] \\
&\hspace{2cm}+ [ (\widetilde{\p}_n(\ng u_n^{\pm} - \ng u^{\pm}))(x) - (\widetilde{\p}_n(\ng u_n^{\pm} - \ng u^{\pm}))(y) ]\,.
\end{align*}
The first term can be rewritten as follows
\begin{align*}
(\widetilde{\p}_n(\lambda_n&-\id)\ng u_n^{\pm})(x) - (\widetilde{\p}_n(\lambda_n-\id)\ng u_n^{\pm})(x) \\
&= (\widetilde{\p}_n(\lambda_n-\id))(x)[\ng u_n^{\pm})(x) - \ng u_n^{\pm})(y)]\\
&\hspace{0.5cm}+ \widetilde{\p}_n(x)[(\lambda_n-\id)(x) - (\lambda_n-\id)(y)]
+ (\widetilde{\p}_n(x)-\widetilde{\p}_n(y))(\lambda_n-\id)(y)\,,
\end{align*}
while the last one as
\begin{align*}
(\widetilde{\p}_n&(\ng u_n^{\pm} - \ng u^{\pm}))(x) - (\widetilde{\p}_n(\ng u_n^{\pm} - \ng u^{\pm}))(y)\\
&= \widetilde{\p}_n(x)[\ng u_n^{\pm} - \ng u^{\pm})(x)-(\ng u_n^{\pm} - \ng u^{\pm})(y))]\\
&\hspace{0.5cm}+[\widetilde{\p}_n(x)-\widetilde{\p}_n(y)][\ng u_n^{\pm} - \ng u^{\pm})(y)]\,.
\end{align*}
Thus the Gagliardo $H^{\frac{1}{2}}$-semi-norm of $\Phi_n$ can be estimated as follows:
\begin{align}\label{eq:stima}
\int_{\g}\int_{\g}&\frac{|\Psi_n(x)-\Psi_n(y)|^2}{|x-y|^2}\dh(x)\dh(y) \nonumber\\
&\leq \|\widetilde{\p}_n\|^2_{C^0(\g)} \|\lambda_n-\id\|^2_{C^0(\bar{\o};\R^{n^2})} \bigl[\, \|\ng u_n^{\pm}\|^2_{H^{\frac{1}{2}}(\g)} + \|\widetilde{\p}_n\|
              ^2_{H^{\frac{1}{2}}(\g)} \,\bigr] \nonumber\\
&\hspace{1cm}+ \mathcal{H}^1(\g) \|\widetilde{\p}_n\|^2_{C^0(\g)} \|\lambda_n-\id\|^2_{C^0(\bar{\o};\R^{n^2})}  \nonumber\\
&\hspace{1cm}+ \|\widetilde{\p}_n\|^2_{C^0(\g)} \|\ng u_n^{\pm} - \ng u^{\pm}\|^2_{H^{\frac{1}{2}}(\g)}  \nonumber\\
&\hspace{1cm}+ \|\widetilde{\p}_n\|^2_{H^{\frac{1}{2}}(\g)} \|\ng u_n^{\pm} - \ng u^{\pm}\|^2_{C^0(\g)}\,.  
\end{align}
To estimate the terms on the right-hand side we will use the following facts:
\begin{itemize}
\item $(\widetilde{\p}_n)_n$ is bounded in $H^{\frac{1}{2}}(\g)$ and in $C^0(\g)$,
\item $\lambda_n\rightarrow\id$ in $C^1(\o;\R^4)$,
\item $u_n\rightarrow u\,\,\,\,\,\,\text{ in }H^2\bigl((\o\meno\g)\cap V\bigr)$, where $V$ i s a neighborhood of $\g$ in $\o$ such that $V\cap\pdo=\emptyset$.
\end{itemize}
Indeed the first fact follows directly from the Sobolev embeddings, since $(\widetilde{\p}_n)_n$ is bounded in $H^1(\g)$,
the second convergence is easy from the fact that $\Phi_n\rightarrow\id$ in $C^2$, while the last claim is a consequence of the continuity property of elliptic boundary value problems: writing the equation satisfied by $u_n$ on $\o$ we notice that the coefficients of the elliptic operator converge
to those of the laplacian. Thus, by Theorem ~\ref{thm:ellcont} we get that $u_n\rightarrow u$ in $H^1(\o)$ and by the estimate ~\eqref{eq:aprst} that the convergence is actually in $H^2((\o\meno\g)\cap V)$ (notice that we have to restrict ourselves to a neighborhood of $\g$ in order to avoid the singularities of $u$ where the Neumann boundary condition transforms into a Dirichlet one).\\
Thus we conclude from \eqref{eq:stima} that $\Psi_n\rightarrow0$ in $H^{\frac{1}{2}}(\g)$.\\

Combining all the above convergence, one gets that
$$
\bigl| \partial^2\ms\bigl( (u_{\Phi_n},\g_{\Phi_n}); U  \bigr)[\p_n] - \partial^2\ms\bigl( (u,\g); U  \bigr)[\widetilde{\p}_n] \bigr| \rightarrow0\,,
$$
and hence, by \eqref{eq:abs}, that
$$
\partial^2\ms\bigl( (u,\g); U  \bigr)[\widetilde{\p}_n] \rightarrow0\,.
$$
But this is in contradiction with the result of Lemma ~\ref{lem:coerc}.
\end{proof}

\subsection{Proof of Theorem ~\ref{thm:minC2}}

We are now ready to prove Theorem ~\ref{thm:minC2}.

\begin{proof}[Proof of Theorem ~\ref{thm:minC2}.] Let $\Phi$ as in the statement of the theorem.

\emph{Step 1}. Suppose $\Phi$ satisfies the following additional hypothesis:
$\Phi\in C^3(\bar{\o};\bar{\o})$, $\Phi(x_0)\neq x_0$ and  $\Phi\bigl( \bar{\g}\cap B_r(x_0) \bigr)= \bar{\g}\cap B_r(x_0)+v$ for some $r>0$ and
$v\in\R^2$.\\

\emph{First step}. Consider the diffeomorphisms $(\Phi_t)_t$ given by Proposition ~\ref{prop:x}. Define the function $g(t):=\ms\bigl((u_t,\g_t);U\bigr)$. Since $(u,\g)$ is a critical point we have that $g'(0)=0$. Hence we can write
\begin{equation*}
\ms\bigl((u_{\Phi}, \g_{\Phi});U\bigr) - \ms\bigl((u,\g);U\bigr) = g(1) - g(0) = \int_0^1 (1-t)g''(t)\,\mathrm{d}t\,.
\end{equation*}
We claim that there exists $\bar{\delta}_1>0$, and a constant $C>0$ such that
\begin{equation}\label{eq:st}
g''(t)\geq C\|X\cdot\nu_t\|_{\widetilde{H}^1(\g_t)}^2\,,
\end{equation}
whenever $\|\Phi-\id\|_{C^2(\bar{\o};\bar{\o})}<\bar{\delta}<\bar{\delta}_1$.
This allows us to conclude. Indeed the local minimality follows directly from \eqref{eq:st}, while the isolated local minimality can be deduced from the fact that $\ms\bigl((u,\g);U\bigr)=\ms\bigl((v,\g_\Phi);U\bigr)$ implies $g''(t)=0$ for each $t\in[0,1]$. In particular $g''(0)=0$, and this implies that $X\cdot\nu_t\equiv0$ on $\g_t$. Looking at the construction of the vector field $X$ (see Proposition ~\ref{prop:cont}) this implies that $X\equiv0$ on $\g$, that is $\g_\Phi=\g$.
Since now the curve $\g$ is fixed and we already know that $u$ minimizes the Dirichlet integral over $\o\meno\g$, we obtain the isolated local minimality of $(u,\g)$ as wanted.\\

Let us now prove \eqref{eq:st}. First of all we notice that, by criticality of $(u,\g)$, $\g$ intersects $\po$ orthogonally and $\nu^i(x_0), \nu^j(x_0)$ are linear independent for $i\neq j$. Thus it is possible to take $\bar{\delta}$ sufficiently small in order to have the that $\g_t$ intersects $\po$ in a non tangent way and that $\nu^i_t(x_0), \nu^j_t(x_0)$ are still linearly independent for $i\neq j$. By the definition of $u_t$ we have that $\partial_{\nu_t} u_t^{\pm}=0$ on $\g_t$ and $\partial_{\nu_{\po}}u^{\pm}=0$ on $(\partial\o\meno\pdo)\cap\bar{U}$. Then
$$
\nabla_{\g}u^{\pm}(x_i^t)=0\quad\quad\text{ for }i=0,1,2,3\,.
$$
In particular $f_t-H_t=0$ on $\partial\g_t$.
Thus, by Remark ~\ref{rem:secvar}, we can write
\begin{align}\label{eq:g}
g''(t) & = \frac{\mathrm{d}^2}{\mathrm{d}s^2} {\ms\bigl((u_s,\g_s);U\bigr)}_{|s=t} = \partial^2\ms\bigl( (u_t,\g_t);U \bigr)[X\cdot\nu_t] \nonumber \\
& + \int_{\g_t}f_t\bigl[Z\cdot\nu_t -2X^{||}\cdot\nabla_{\g_t}(X\cdot\nu_t)+D\nu_t[X^{||},X^{||}] - H_t(X\cdot\nu_t)^2 \bigr]\dh \nonumber \\
& + \sum_{i=1}^3 (X\cdot\nu_t)^2D\nu_{\po}[\nu_t,\nu_t](x_i^t)+\int_{\partial\g_t}Z\cdot\eta_t\,\dc,
\end{align}
where we set $x_i^t:=\Phi_t(x_i)$ for $i=0,1,2,3$.
Now we need to estimate each of the above terms. Fix $\zeta>0$. For the first one we appeal to Proposition ~\ref{prop:cont} to obtain
\begin{equation}\label{eq:st1}
\partial^2\ms\bigl( (u_t,\g_t);U \bigr)[X\cdot\nu_t] \geq \bar{C} \|X\cdot\nu_t\|^2_{\widetilde{H}^1(\g_t)}\,.
\end{equation}

To estimate the second term of \eqref{eq:g} we recall that Remark ~\ref{rem:conv} and the continuity of the map $\Phi\mapsto H_{\Phi}$ assert that the map
$$
\Phi\in\mathcal{D}_{\bar{\delta}_1}(\o;U) \mapsto \bigl\| |\nabla_{\g_{\Phi}}u_{\Phi}^+|^2 - |\nabla_{\g_{\Phi}}u_{\Phi}^-|^2 + H_{\Phi} \bigr\|_{L^{\infty}(\g_\Phi)}
$$
is continuous with respect to the $C^2$-norm. Since by the criticality condition that quantity vanishes for $\Phi=\id$, possibly reducing $\bar{\delta}_1$, it is possible to have
$$
\bigl\| |\nabla_{\g_{\Phi}}u_{\Phi}^+|^2 - |\nabla_{\g_{\Phi}}u_{\Phi}^-|^2 + H_{\Phi} \bigr\|_{L^{\infty}(\g_\Phi)} \leq \zeta\,,
$$
for each $\Phi\in\mathcal{D}_{\bar{\delta}_1}(\o;U)$. Hence
\begin{align}\label{eq:st2}
\int_{\g_t} &f_t\bigl[Z\cdot\nu_t -2X^{||}\cdot\nabla_{\g_t}(X\cdot\nu_t)+D\nu_t[X^{||},X^{||}] - H_t(X\cdot\nu_t)^2 \bigr]\dh \nonumber \\
& \geq -\zeta \|Z\cdot\nu_t -2X^{||}\cdot\nabla_{\g_t}(X\cdot\nu_t)+D\nu_t[X^{||},X^{||}] - H_t(X\cdot\nu_t)^2 \|_{L^1(\g_t)} \nonumber \\
&\geq C \|X\cdot\nu_t\|^2_{\widetilde{H}^1(\g_t)}\,,
\end{align}
where in the last step we used estimates \eqref{eq:stima1} and \eqref{eq:stima2} provided by Proposition~\ref{prop:x}.

To estimate the last term we recall that $Z\equiv0$ in a neighborhood of $x_0$.
Thus, we can rewrite the last term as
\begin{align*}
\sum_{i=1}^3 (X\cdot\nu_t)^2D\nu_{\po}&[\nu_t,\nu_t](x_i^t)+\int_{\partial\g}Z\cdot\eta_t\,\dc\\
&=\sum_{i=1}^3 \bigl[(X\cdot\nu_t)^2D\nu_{\po}[\nu_t,\nu_t] + Z\cdot\eta_t \bigr](x_i^t)\\
&= \sum_{i=1}^3 \bigl[(X\cdot\nu_t)^2D\nu_{\po}[\nu_t,\nu_t] + Z\cdot(\eta_t-\nu_{\po}) + Z\cdot\nu_{\po} \bigr](x_i^t)\\
&= \sum_{i=1}^3 \bigl[-(X\cdot\eta_t)^2D\nu_{\po}[\eta_t,\eta_t] + Z\cdot(\eta_t-\nu_{\po}) \bigr](x_i^t)\,,
\end{align*}
where we have used equality \emph{(iii)} of Lemma ~\ref{lem:geom}.
We claim that it is possible to choose $\bar{\delta}_1$ in such a way that
\begin{equation}\label{eq:cond1}
|X\cdot\eta_t(x_i^t)|^2\leq\zeta|X\cdot\nu_t(x_i^t)|^2\,,
\end{equation}
\begin{equation}\label{eq:cond2}
|\eta_t-\nu_{\po}|(x_i^t)\leq\zeta\,,
\end{equation}
and
\begin{equation}\label{eq:cond3}
|Z(x_i^t)|\leq C\|X\cdot\nu_t\|^2_{H^1(\g^i)}\,,
\end{equation}
for all $i=1,2,3$. Indeed, \eqref{eq:cond2} follows easy by noticing that $\eta(x_i)=\nu_{\po}(x_i)$ and by the identity
$$
\eta_t=\frac{D\Phi_t[\eta]}{|D\Phi_t[\eta]|}\,.
$$
To obtain \eqref{eq:cond1} we notice that from $X=\nu$ on $\po\cap\partial\g$ we get $X\cdot\eta(x_i)=0$. Then we conclude thanks to the continuity of the maps
$$
G_i: \bigl\{(x,v,w)\in \bigl(\po\cap B_{\bar{\delta}}(x_i)\bigr)\times \bigl(B_{\bar{\delta}}(\eta(x))\cap S^1 \bigr)
\times \bigl( B_{\bar{\delta}}(\nu(x))\cap S^1\bigr) \bigr\}\rightarrow\R\,,
$$
given by
$$
G_i(x,v,w):=\frac{|F(x)\cdot v|}{|F(x)\cdot w|}\,.
$$
Finally, in order to obtain \eqref{eq:cond3}, we notice that, by construction of the vector field $X$, there exists a function $\Phi\in C^2((\g)_{\bar{\delta}})$ that is constant along the trajectories of $F$, such that $X=\Phi F$ near $\po$ (see Proposition ~\ref{prop:x}). Hence
$$
Z(x_i^t)=DX[X](x_i^t)=\Phi(x_i)^2 DF[F](x_i^t)\,.
$$
Reasoning in a similar way as above, taking a $\bar{\delta}_1$ sufficiently small, we have that $|F\cdot\nu_t(x_i^t)|\geq\frac{1}{2}$, and hence $\Phi(x_i)^2\leq 2(X\cdot\nu_t)(x)^2$ for $x\in B_{\bar{\delta}}(x_i)$. Thus, we obtain the estimate
$$
|Z(x_i^t)|\leq C|X\cdot\nu_t|^2\leq C_2\|X\cdot\nu_t\|_{\widetilde{H}^1(\g)}^2\,,
$$
where $C>0$ depends only on $F$ and $\bar{\delta}$ and the last inequality follows by the Sobolev embedding.
Using \eqref{eq:cond1}, \eqref{eq:cond2} and \eqref{eq:cond3} we obtain
\begin{equation}\label{eq:st3}
\sum_{i=1}^3 \bigl[-(X\cdot\eta_t)^2D\nu_{\po}[\eta_t,\eta_t] + Z\cdot(\eta_t-\nu_{\po}) \bigr](x_i^t)\geq-C_2\zeta\|X\cdot\nu_t\|_{\widetilde{H}^1(\g)}^2\,.
\end{equation}\\
Now, combining the estimates \eqref{eq:st1}, \eqref{eq:st2} and \eqref{eq:st3}, we get that:
$$
\int_0^1 g''(t) \mathrm{d}t \geq \bigl(\bar{C}-(C_1+C_2)\zeta\bigr) \int_0^1 \|X\cdot\nu_t\|_{\widetilde{H}^1(\g)}^2 \mathrm{d}t\,.
$$
Thus, by taking $\zeta$ sufficiently small, we finally have the claimed bound \eqref{eq:st}.\\

\emph{Step 2}. It is easy to see that, given $\Phi$ as in Step 1, but with $\Phi(x_0)=x_0$, it is possible to construct a family of diffeomorphisms
$\Psi_\varepsilon:\Phi(\bar{\g})\rightarrow\bar{\o}$ such that $\Psi_\varepsilon(x_0)\neq x_0$ and $\Psi_\varepsilon\rightarrow\id$ in the $C^2$ norm,
as $\varepsilon\rightarrow0$. This implies that
$$
\ms\bigl((u_{\Psi_\varepsilon}, \Psi_\varepsilon(\Phi(\g));U\bigr)\rightarrow \ms\bigl((u_{\Phi}, (\Phi(\g);U\bigr)\,.
$$
Thus the result follows by passing to the limit in the inequality proved in the previous case.\\

\emph{Step 3}. Notice that all the previous steps have been done just by using the closeness of $\Phi$ to the identity in the $C^2$-norm. So, given a diffeomorphism $\Phi\in C^2(\bar{\o};\bar{\o})$ such that
$\|\Phi-\id\|_{C^2(\g;\bar{\o})}<\bar{\delta}$, we can find
$(\Phi_\varepsilon)_\varepsilon\subset C^3(\bar{\o};\bar{\o})$ with $\Phi_\varepsilon=\id$ in $\pdo\cup(\o\meno U')$, where $U\subset U'$, with $U'$ an admissible subdomain, such that $\Phi_\varepsilon\rightarrow\Phi$ in $C^2(\bar{\o};\bar{\o})$.
Using Remark ~\ref{rem:monot}, we know that, if $U'$ is close to $U$ in the Hausdorff sense, then $(u,\g)$ is stable also in $U'$. The result follows by passing to the limit.\\

\emph{Step 4}. Finally, the local minimality with respect to $W^{2,\infty}$-perturbations can be obtained by approximating an admissible diffeomorphism
$\Phi\in W^{2,\infty}(\bar{\o};\bar{\o})$ with a sequence of diffeomorphisms of class $C^2$ converging to $\Phi$ in the $W^{2,\infty}$-topology.\\

Finally, the isolated local minimality follows by Remark~\ref{rem:cont}.
\end{proof}


\section{Application}

In this section we would like to give some examples of critical and strictly stable triple points.

\subsection{Local minimality in a tubular neighborhood}

Here we want to prove that, under an additional assumption (similar to those of \cite{CagMorMor} and \cite{BonMor}), every critical triple point is strictly stable in a suitable tubular neighborhood, and hence a local minimizer with respect to $W^{2,p}$-variations contained in that tubular neighborhood.

\begin{proposition}
Let $(u,\g)$ be a critical triple point, and suppose that
\begin{equation}\label{eq:condpos}
H_{\partial\o}(x_i)<0\,,
\end{equation}
for each $i=1,2,3$. Then there exists $\bar{\mu}>0$ such that, for all $\mu<\bar{\mu}$, $(u,\g)$ is strictly stable in $(\g)_\mu$.
\end{proposition}

\begin{proof}
\emph{Step 1}. First of all we prove that there exists a constant $C>0$ such that:
$$
\int_{\g} |\nabla_{\g}\p|^2\,\dh\,+\,\int_{\g}H^2\p^2\,\dh\,-\,\sum_{i=1}^3\p_i^2(x_i)H_{\partial\o}(x_i)\geq C\|\p\|^2_{\widetilde{H}^1(\g)}\,,
$$
for all $\p\in\widetilde{H}^1(\g)$. Indeed, by using the Sobolev embeddings, it is easy to see that
$$
\int_{\g^i} |\nabla_{\g}\p_i|^2\,\dh-\,\p_i^2(x_i)H_{\partial\o}(x_i)\geq C\int_{\g^i}|\p_i|^2\dh\,.\\
$$

\emph{Step 2}. The only thing we have to prove now is that
\begin{equation}\label{eq:conve}
\lim_{\mu\rightarrow0}\, \sup_{\substack{\p\in\widetilde{H}^1(\g),\\ \|\p\|_{\widetilde{H}^1(\g)}=1}}
                          \int_{(\g)_\mu} |\nabla v^\mu_\p|^2\,\mathrm{d}x = 0\,,
\end{equation}
where $v^\mu_{\p}\in H^1_{(\g)_\mu}(\o\meno\g)$ is the solution of
\begin{equation}\label{eq:v}
\int_{\o}\nabla v^\mu_{\p}\cdot\nabla z\,\mathrm{d}x = \langle \div_{\g}\bigl( \p\nabla_{\g}u^+ \bigr), z^+ \rangle_{H^{-\frac{1}{2}}(\g)\times H^{\frac{1}{2}}(\g)} - \langle \div_{\g}\bigl( \p\nabla_{\g}u^- \bigr), z^- \rangle_{H^{-\frac{1}{2}}(\g)\times H^{\frac{1}{2}}(\g)}\,,
\end{equation}
for every $z\in H^1_{(\g)_\mu}(\o\meno\g)$. For each $\mu>0$, let $\bar{\p}^\mu\in\widetilde{H}^1(\g)$, with $\|\bar{\p}^\mu\|_{\widetilde{H}^1(\g)}=1$, be such that
$$
\int_{(\g)_\mu} |\nabla v^\mu_{\bar{\p}^\mu}|^2\,\mathrm{d}x\,=\,\sup_{\substack{\p\in\widetilde{H}^1(\g),\\ \|\p\|_{\widetilde{H}^1(\g)}=1}}
                          \int_{(\g)_\mu} |\nabla v^\mu_\p|^2\,\mathrm{d}x\,.
$$
Consider, for $\mu>0$, the following minimum problem: 
$$
\min\{F^\mu(v)\,:\, v\in H^1_{(\g)_\mu}(\o\meno\g) \}\,,
$$
where we define the functional
\begin{align*}
F^\mu(v)&:=\frac{1}{2}\int_{(\g)_\mu}|\nabla v|^2\,\mathrm{d}x\,-\, \langle \div_{\g}\bigl( \bar{\p}^\mu\nabla_{\g}u^+ \bigr), v^+ \rangle_{H^{-\frac{1}{2}}(\g)\times H^{\frac{1}{2}}(\g)}\\
&\hspace{0.5cm}+ \langle \div_{\g}\bigl( \bar{\p}^\mu\nabla_{\g}u^- \bigr), v^- \rangle_{H^{-\frac{1}{2}}(\g)\times H^{\frac{1}{2}}(\g)}\,.
\end{align*}
Since the equation defining $v^\mu_{\bar{\p}^\mu}$ is the Euler-Lagrange equation for $F^\mu$, we have that $v^\mu_{\bar{\p}^\mu}$ is the unique solution of the above minimum problem. We claim that
\begin{equation}\label{eq:stf}
F^\mu(v)\geq \frac{1}{4}\int_{(\g)_\mu}|\nabla v|^2\,\mathrm{d}x\,-C\,,
\end{equation}
for a suitable constant $C>0$. Taking \eqref{eq:stf} for grant, we conclude. Indeed, noticing that
\begin{equation}\label{eq:monot}
\min\{F^\mu(v)\,:\, v\in H^1_{(\g)_\mu}(\o\meno\g) \}=-\frac{1}{2}\int_{(\g)_\mu}|\nabla v^\mu_{\bar{\p}^\mu}|^2\,\mathrm{d}x\,,
\end{equation}
from \eqref{eq:stf} we get that
$$
\sup_{\mu>0} \int_{(\g)_\mu}|\nabla v^\mu_{\bar{\p}^\mu}|^2\,\mathrm{d}x \leq M\,,
$$
for some $M>0$. So, up to a not relabelled subsequence, $v^\mu_{\bar{\p}^\mu}\rightharpoonup w$ weakly in $H^1(\o\meno\g)$, as $\mu\rightarrow0$. It is easy to see that $w=0$. Then, using equation \eqref{eq:v} where we take as a test function $\nabla v^\mu_{\bar{\p}^\mu}$ itself, the uniform bound on $\|\div_{\g}\bigl( \bar{\p}^\mu\nabla_{\g}u^\pm \bigr)\|_{H^{-\frac{1}{2}}(\g)}$, and the compactness of the trace operator, we finally get:
$$
\int_{(\g)_\mu}|\nabla v^\mu_{\bar{\p}^\mu}|^2\,\mathrm{d}x \rightarrow0\,,\quad\text{ as } \mu\rightarrow0.
$$
We are now left to prove estimate \eqref{eq:stf}. Fix $\bar{\mu}>0$ and let $\Phi_\mu:=\bar{\p}^\mu\nabla_{\g}u^+$. Then:
\begin{align*}
\langle \div_{\g}\bigl( \bar{\p}^\mu\nabla_{\g}u^\pm \bigr), v^+ \rangle_{H^{-\frac{1}{2}}(\g)\times H^{\frac{1}{2}}(\g)} & \leq 
               \|\div_{\g}\Phi_\mu\|^2_{H^{-\frac{1}{2}}(\g)} \|v^\pm\|^2_{H^{\frac{1}{2}}(\g)} \\
& \leq C\frac{\varepsilon^2}{2}\|v\|_{H^1((\g)_{\bar{\mu}})} + \frac{C}{2\varepsilon^2} \\
& \leq C\frac{\varepsilon^2}{2}\|\nabla v\|_{L^2((\g)_{\bar{\mu}})} + \frac{C}{2\varepsilon^2}\,.
\end{align*}
Thus, taking $\varepsilon>0$ sufficiently small, we obtain the desired estimate.
\end{proof}

\begin{remark}\label{rem:monot}
In view of the result of Theorem ~\ref{thm:minC2}, it is not restrictive to suppose an admissible set $U$ to be of class $C^\infty$, meeting $\partial\o$ orthogonally. Indeed, by \eqref{eq:monot}, it follows that 
$$
-\int_{U_1}|\nabla v_\p|^2\,\mathrm{d}x \geq -\int_{U_2}|\nabla v_\p|^2\,\mathrm{d}x\,,
$$
for every $\p\in\widetilde{H}^1(\g)$, whenever $U_1$ and $U_2$ are admissible subdomains such that $U_1\subset U_2$. Hence, given a regular critical and strictly stable triple point $(u,\g)$ and generic admissible subdomain $U$, we can write $U=\bigcap_n U_n$, with $U_n$ admissible subdomains where $(u,\g)$ is strictly stable, that are of class $C^\infty$ meeting $\partial\o$ orthogonally.
\end{remark}




\section{Appendix}

\subsection{Results on elliptic problems} 

The following theorem collects some regularity results on elliptic problems in domains with corners we will need in the following.
All these results can be found in the book of Grisvard (see \cite{Gris}).\\

\textbf{Notation}. In this section we will consider operators $L$ written in the form
$$
Lu = -\sum_{i,j=1}^2 D_i(a_{ij}D_ju) + \sum_{i=1}^2 a_i D_iu + a_0u\,,
$$
where $D_i$ denotes the partial derivatives with respect to the variable $x_i$.

\begin{definition}
We say that an open and bounded set $A\subset\R^2$ is a \emph{curvilinear polygon} of class $C^{r,s}$, with $r\in\N$ and $s\in(0,1]$, if the following are satisfied
\begin{itemize}
\item[(i)] $\partial A$ is a simple and connected curve that can be written as
$$
\partial A=\cup_{i=1}^k\bar{\gamma}_i\,,
$$
where each $\gamma_i$ is a open curve of class $C^{r,s}$ up to its closure, and $k\in\N$,
\item[(ii)] denoting by $P_i$ the common boundary point of $\gamma_i$ and $\gamma_{i+1}$ (and by $P_k$ the common one of $\gamma_k$ and $\gamma_1$), and by $\omega_i$ the angle in $P_i$ internal to $A$, we have $\omega_i\in(0,2\pi)$.
\end{itemize}
\end{definition}

\begin{theorem}\label{thm:regu}
Let $A$ be  a curvilinear polygon of class $C^{1,1}$, and let $L$ be an elliptic operator defined on $A$, with coefficients of class $C^{0,1}$.
Then, the following a priori estimate holds true:
\begin{equation}\label{eq:aprst}
\|u\|_{H^2(A)}\leq C_1\bigl( \|Lu\|_{L^2(A)} + \|\partial_{\nu}u\|_{H^{\frac{1}{2}}(\partial A)} + \|u\|_{H^{\frac{3}{2}}(\partial A)}  \bigr) + C_2\|u\|_{H^1(A)}\,,
\end{equation}
for suitable constants $C_1, C_2\geq 0$ and for all $u\in H^2(A)$.\\

Given $f\in L^2(A)$, let $u\in H^1(A)$ be a weak solution of the problem
$$
\left\{
\begin{array}{ll}
Lu=f & \text{ in } A\,,\\
\partial_{\nu}u = 0 & \text{ on } \gamma_i\,,\text{ for } i\in\mathcal{N}\,,\\
u = 0 & \text{ on } \gamma_i\,,\text{ for } i\in\mathcal{D}\,,\\
\end{array}
\right.
$$
where $\mathcal{N},\mathcal{D}$ is a partition of $\{1,\dots,N\}$. Then $u$ can be written as
$$
u=u_{reg}+\sum_{i=1}^{N}u^i_{sing}\,,
$$
where $u_{reg}\in H^2(A)$ and $u^i_{sing}\in H^1(A)$ are such that $u^i_{sing}\in H^2(V_i)$ for each open set $V_i$ such that $P_i\not\in\bar{V_i}$.\\

Finally, suppose that
$$
\omega_i\leq\left\{
\begin{array}{ll}
\pi & \text{ if } j,j+1\in\mathcal{D}\,,\text{ or } j,j+1\in\mathcal{N}\,,\\
\frac{\pi}{2} & \text{ otherwise }\,.
\end{array}
\right.
$$
Then $u\in H^2(A)$. Moreover, if $\mathcal{D}$ is not empty, \eqref{eq:aprst} holds with $C_2=0$.
\end{theorem}

Finally, we need a continuity theorem for elliptic problems (for a proof see, \emph{e.g.}, \cite[Remark~2.2]{Lazz}).

\begin{theorem}\label{thm:ellcont}
Let $(L_s)_{s\in(-\delta,\delta)}$ be a family of uniformly elliptic operators defined on a curvilinear polygon $A$ of class $C^{1,1}$, and let $\mathcal{N},\mathcal{D}$ be a partition of $\{1,\dots,N\}$, with $\mathcal{D}\neq\emptyset$. Suppose that, for $s\in(-\delta,\delta)$, the functions
$$
s\mapsto a_{ij}(\cdot,s)\,\quad\quad s\mapsto a_i(\cdot,s)\,,\quad\quad s\mapsto a_0(\cdot,s)\,,
$$
and
$$
s\mapsto f_s\,\in H^{-1}(A)
$$
are continuous and that there exists a constant $M>0$ such that
$$
|a_{ij}(x,s)|\leq M\,,\quad\quad |a_i(x,s)|\leq M\,,\quad\quad |a_0(x,s)|\leq M \,,
$$
for all $s\in (-\delta,\delta)$ and for a.e. $x\in A$ and . Given $v\in H^1(A)$ let us consider the operator
$$
\begin{array}{cccc}
T: & (-\delta,\delta) & \rightarrow & H^1 \\
   & s & \mapsto & u_s
\end{array}
$$
where $u_s$ is a weak solution of the problem
$$
\left\{
\begin{array}{ll}
L_su=f_s & \text{ in } A\,,\\
\partial_{\nu}u = 0 & \text{ on } \gamma_i\,,\text{ for } i\in\mathcal{N}\,,\\
u_s = v & \text{ on } \gamma_i\,,\text{ for } i\in\mathcal{D}\,.\\
\end{array}
\right.
$$
Then $T$ is continuous.
\end{theorem}


\subsection{Extension results}

We start by stating the version of the Whitney's extension theorem needed in the proof of Proposition \ref{prop:pos}.
We first need to set some notation. Given $\textbf{k}=(k_1,\dots,k_n)\in\mathbb{N}^2$ and $v\in\R^2$, let
$$
|\textbf{k}|:=k_1+k_2\,,\quad\quad v^{\textbf{k}}:=v_1^{k_1}v_2^{k_2}\,.
$$
If $f$ is a $|\textbf{k}|$-times differentiable function, we set
$$
D^{\textbf{k}}f(x):=\frac{\partial^{|\textbf{k}|}f}{\partial x^{\textbf{k}}}(x)=\frac{\partial^{|\textbf{k}|}f}{\partial x_1^{k_1}x_2^{k_2}}(x)\,,
$$
where $D^{\textbf{0}}=f$.

\begin{definition}\label{def:regX}
Let $X$ be a compact subset of $\R^2$. We define the space $C^h(X)$ as the space of functions $f:X\rightarrow\R$ for which there exists a family
$\mathcal{F}:=\{F^{\textbf{k}}\}_{|\textbf{k}|\leq h}$ of continuous functions on $X$, with $F^{\textbf{0}}=f$, such that, for every $|\textbf{k}|\leq h$, it holds
\begin{equation}\label{eq:W}
\sup_{x,y\in X,\, 0<|x-y|<r}\,\Bigl| F^{\textbf{k}}(x)-F^{\textbf{k}}(y) - \sum_{|\textbf{j}|=1}^{h-|\textbf{k}|} F^{\textbf{j}}(x)(y-x)^{\textbf{k}+\textbf{j}} \Bigr|=o(r^{h-|\textbf{k}|})\,.
\end{equation}
Moreover, we define
$$
\|\mathcal{F}\|_{C^h(X)}:=\sum_{|\textbf{k}|\leq h}\|F^{\textbf{k}}\|_{C^0(X)}\,.
$$
\end{definition}

\begin{theorem}[Whitney's extension theorem]
For every $h\geq1$ and $L>0$ there exists a constant $C_0>0$, depending on $h$ and $L$, with the following property: if $X\subset B_L$ is a compact set of $\R^2$ and $f\in C^h(X)$, then there exists a function $\widetilde{f}\in C^{\infty}(\R^2\meno X)\cap C^{h}(\R^2)$ such that
$$
D^{\textbf{k}}\widetilde{f} = F^{\textbf{k}}\quad\text{on } X\,,\,\,\,\text{ for every } |\textbf{k}|\leq h\,,
$$
and
$$
\|\widetilde{f}\|_{C^h(\R^2)}\leq C_0\|\mathcal{F}\|_{C^h(X)}\,.\\
$$
\end{theorem}

We now present a technical result regarding the extension of a function defined on the boundary of a set to functions defined in the whole set.
The \emph{raison d'\^{e}tre} of this result (instead of making use of the Whitney extension Theorem) is because we will apply it in \ref{prop:x}, where we are not able to provide the estimates at $x_0$ needed in order to apply Whitney's theorem. Moreover we also need to control the behavior of the functions $t\mapsto\Phi_t(x)$ for $x\in\o$ and this turns out to be more clear by using our extension result.

\begin{lemma}\label{lem:ext}
Let $A\subset\R^2$ be an open bounded set whose boundary is a curvilinear polygon of class $C^k$.
Let us write $\partial A=\cup_{i=1}^N\bar{\gamma}^i$, with $\gamma^i$ open curve of class $C^k$.
Assume that for each $i=1,\dots,N$ we are given a function $f^i:\gamma^i\rightarrow\R^2$ of class $C^k$ in such a way that $\cup_{i=1}^N f^i(\gamma^i)$ turns out to be a curvilinear polygon of class $C^k$. Let us suppose that each internal angle of $\partial A$ and of $\partial B$ is less than or equal to $\pi$, where $B\subset\R^2$ is the open set whose boundary is given by $\cup_{i=1}^N f^i(\gamma^i)$. Finally, let $g:A\rightarrow B$ be a $C^{k-1}$ function and let $K\subset\o$ be a compact set.
Then there exists a function $f:\bar{A}\rightarrow\R^2$ such that
\begin{itemize}
\item[(i)] $f\in C^{k-1}\bigl(A \cup_{i=1}^N\gamma^i \bigr)$,
\item[(ii)] $f=f^i$ on $\gamma^i$,
\item[(iii)] $f=g$ on $K$.
\end{itemize}  
\end{lemma}

\begin{proof}
It is clear that the technical difficulties are only due to the presence of the corners in the boundary, because otherwise we can just use a standard extension. Thus the idea is to provide a local extension near each corner point and then to glue it with a standard extension in the rest of the boundary.\\
So, let us concentrate in a neighborhood of a corner point $P$ (that is, in what follows we will always suppose everything to be done in a small ball around $P$), and let us denote by $\gamma^i$ and $\gamma^j$ the two curves meeting in $P$. We are now going to describe a construction of a vector field in a tubular neighborhood of $\gamma^j$ that will be subsequently applied to other (couple of) curves. In what follows the normal vector field to the curve will be the one pointing inside the set $A$ or $B$.

For each point $x\in\gamma^j$ let us consider the curve $\gamma^i+(x-P)$, that is the curve $\gamma^i$ translated in $x$.
By the implicit function theorem there exists $\delta>0$ such that for each $y\in(\gamma^j)_\delta$ (the $\delta$-tubular neighborhood of $\gamma^j$ intersected with the set $A$) there exists a unique point $\pi^j_i(y)\in\gamma^j$ such that $y$ belongs to the curve $\gamma^i+(x-P)$. This is possible because $|\tau^i(P)\cdot\nu^j(P)|>0$ and we are working only in a ball around the point $P$. Notice that the map $y\mapsto \pi^j_i(y)$ is of class $C^{k-1}$.
Let us define the vector field $V^j:(\gamma^j)_\delta\rightarrow\R^2$ by
\[
V^j(y):=\chi_\mu\bigl( \pi^j_i(y) \bigr) \widetilde{\tau}^i(y) +
\bigl(\, 1-\chi_\mu\bigl(\pi^j_i(y)\bigr) \,\bigr)\nu^j(y)\,,
\]
where $\chi_\mu(x):=\chi\Bigl( \frac{|x-P|^2}{\mu^2} \Bigr)$ and with $\widetilde{\tau}^i(y)$ we denote the vector
$\tau^i\bigl(y-(x-P)\bigr)$. By construction, the vector field $V^j$ is of class $C^{k-1}$.
By using the implicit function theorem we obtain the existence of $\delta>0$ such that for each $y\in(\gamma^j)_\delta$ there exist a unique $\pi^V_j(y)\in\gamma^j$ whose trajectory along the vector field $V^j$ passes through $y$. Notice that the map $\pi^V_j$ is of class $C^{k-1}$.


We will apply the above construction, \emph{mutatis mutandis}, also to the curve $\gamma^i$ and to the couple of curves
$f^i(\gamma^i)$ and $f^j(\gamma^j)$ and we will denote by $\widetilde{\pi}^V_i$ and $\widetilde{\pi}^V_j$ respectively the projections on $f^i(\gamma_i)$ and on $f^j(\gamma^j)$. Let $\delta>0$ such that all the above projections are well defined in the $\delta$-tubular neighborhood of each curve. Let us denote by $C^{ij}$ the region of $f(A)$ where both the projections $\widetilde{\pi}^V_j$ and $\widetilde{\pi}^V_i$ are defined. There we can define the projection
$\pi: C^{ij}\rightarrow f^i(\gamma^i)\times f^j(\gamma^j)$ as $\pi(y):=\bigl(\, \widetilde{\pi}^V_i(y),\, \widetilde{\pi}^V_j(y) \,\bigr)$. Notice that $\pi$ is of class $C^{k-1}$.
Finally we denote by $N_\delta^i$ the region of $(\gamma^i)_\delta$ where the vector field $V^i$ is the normal vector field $\nu^i$ and by $N_\delta^j$ the region of $(\gamma^j)_\delta$ where the vector field $V^j$ is the normal vector field $\nu^j$.

We can now define our extension as follows: for $y\in(\gamma^i)_\delta\cup(\gamma^j)_\delta$, let
\begin{align*}
f(y)&:=\chi_\lambda(y)\pi^{-1}\bigl(\, f^i\bigl( \pi^V_i(y)\bigr), f^j\bigl( \pi^V_j(y) \bigr) \bigr) \,\bigr)\\
&\hspace{1cm}+(1-\chi_\lambda(y))\Bigl[ \mathbbm{1}_{N^i_\delta}(y)\bigl[ f^i\bigl( \pi^V_i(y)\bigr) + \mathrm{d}(y,\gamma^i)\widetilde{\nu}^i\bigl( f^i\bigl( \pi^V_i(y)\bigr) \bigr) \bigr]\\
&\hspace{1cm}+ \mathbbm{1}_{N^j_\delta}(y)\bigl[ f^j\bigl( \pi^V_j(y)\bigr) + \mathrm{d}(y,\gamma^j)\widetilde{\nu}^j\bigl( f^j\bigl( \pi^V_j(y)\bigr) \bigr) \bigr]   \Bigr]\,,
\end{align*}
where with $\widetilde{\nu}^j$ and $\widetilde{\nu}^j$ we denote the normal vector fields to $f^i(\gamma^i)$ and $f^j(\gamma^j)$ respectively. Notice that $f$ agrees with $f^i$ and $f^j$ on $\gamma^i$ and $\gamma^j$ respectively.
Moreover, the function $f$ is, by construction, of class $C^{k-1}$.

\begin{figure}[H]
\includegraphics[scale=0.6]{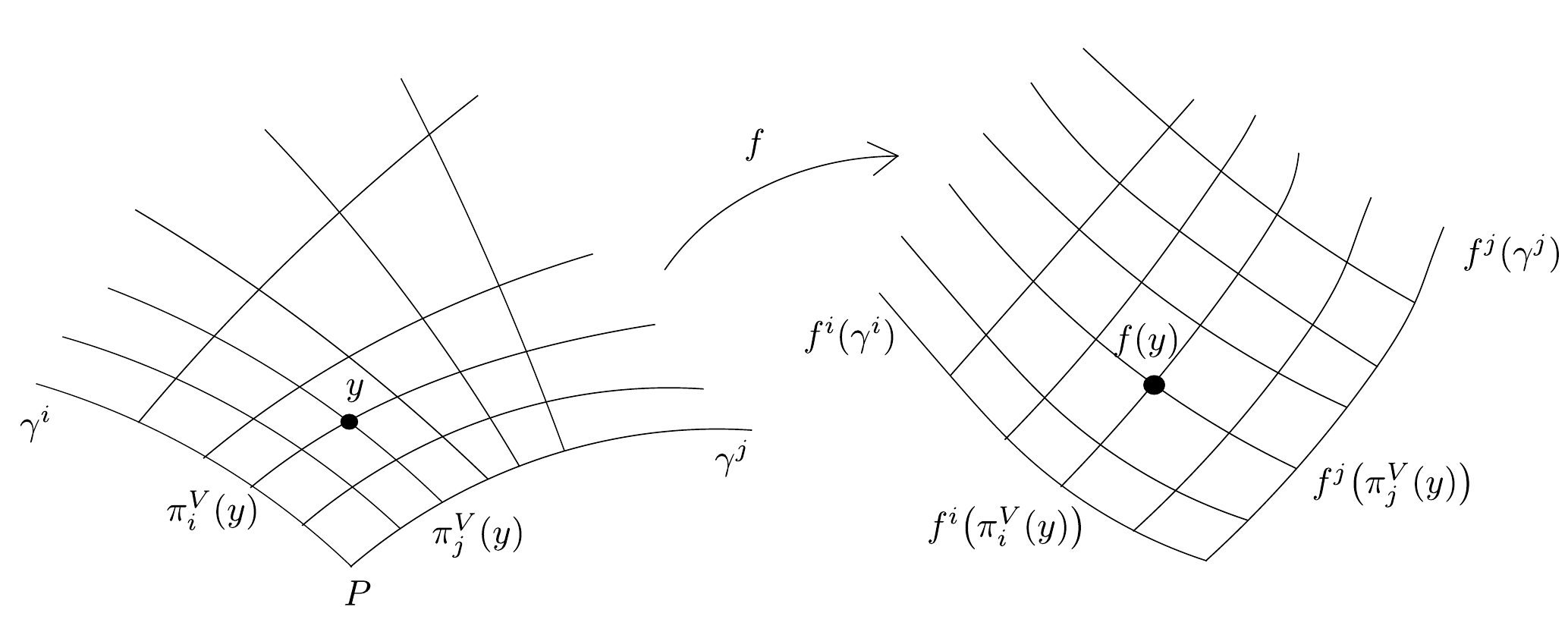}
\caption{The idea of the construction of the extension $f$.}
\label{fig:f}
\end{figure}

We repeat this procedure for all the corner points of $\partial A$, obtaining a function
$\widetilde{f}:(\partial A)_\delta\rightarrow B$.
We have now to glue $f$ with $g$. Let $\bar{\delta}:=\min\{\, \delta, \mathrm{d}(K,\partial A) \,\}$ and take a smooth curve $\widetilde{\gamma}\subset A\cup(\delta A)_{\frac{\bar{\delta}}{2}}$
that does not touch $\partial A$, and let us denote by $\mathrm{d}(\cdot,\widetilde{\gamma})$ the signed distance from $\widetilde{\gamma}$, where  is oriented in such a way that the normal points inside $A$. We then define the function $f$ as
\[
f(y):=\chi\Biggl( \frac{\mathrm{d}^2(y,\widetilde{\gamma})}{\bigl(\frac{\delta}{2}\bigr)^2} \Biggr)\widetilde{f}(y)
+ \Biggl(1-\chi\Biggl( \frac{\mathrm{d}^2(y,\widetilde{\gamma})}{\bigl(\frac{\delta}{2}\bigr)^2} \Biggr)\Biggr)g(y)\,.
\]
\end{proof}


\subsection{Technical results}

We now prove a technical result.

\begin{lemma}\label{lem:3W}
Let $\gamma\subset\o$ be a simple curve of class $C^1$ meeting $\partial\o$ orthogonally in a point $\bar{x}$. Let $X\in C^1(\gamma;\R^2)$ be such that
$X(\bar{x})=\tau_{\partial\o}(\bar{x})$. Then, the vector field defined as
$$
\widetilde{X}:=
\left\{
\begin{array}{ll}
X & \text{ on } \gamma\,,\\
\tau_{\partial\o} & \text{ on } \partial\o\,,
\end{array}
\right.
$$
belongs to $C^1(\bar{\g}\cup\partial\o;\R^2)$.
\end{lemma}

\begin{proof}
Denote by $\tau$ the tangent vector field on $\gamma$. Define
$$
D\widetilde{X}(x)[\tau(x)]:=DX(x)[\tau(x)]\,,\quad
      D\widetilde{X}(x)[\nu(x)]:=\chi\Bigl( \frac{|x-\bar{x}|^2}{\varepsilon^2} \Bigr)D\tau_{\partial\o}(\bar{x})[\tau_{\partial\o}(\bar{x})]\,,
$$
for $x\in\gamma$, and
$$
D\widetilde{X}(x)[\tau_{\partial\o}(x)]:=D\tau_{\partial\o}(x)[\tau_{\partial\o}(x)]\,,\quad
      D\widetilde{X}(x)[\nu_{\partial\o}(x)]:=\chi\Bigl( \frac{|x-\bar{x}|^2}{\varepsilon^2} \Bigr)DX(\bar{x})[\tau(\bar{x})]\,,
$$
for $x\in\partial\o$, for a constant $\varepsilon>0$. Then condition \eqref{eq:W} is easily satisfied if $x,y\in\gamma$ or $x,y\in\partial\o$. In the case $x\in\gamma$ and $y\in\partial\o$, we simply write $y-x=(y-\bar{x})-(x-\bar{x})$ and we use the triangular inequality we get the desired estimate.
\end{proof}

We now prove the necessity of non-negativity of the quadratic form for local minimizers.

\begin{proof}[Proof of Proposition ~\ref{prop:pos}]\label{sec:prop}
\emph{Step 1}. Suppose the statement holds true for all functions $\p\in\widetilde{H}^1(\g)$ such that
$\p^i\in C^\infty(\bar{\g}_i)$ for $i=1,2,3$. Then the result follows.
Indeed, fix $\p\in\widetilde{H}^1(\g)$ and consider approximation by convolution $(\p^i_\varepsilon)_\varepsilon$ of each $\p_i$ (where we have previously extended each $\p^i$ to an $H^1$ function defined in a regular extension of $\g_i$). Now let $h_\varepsilon:=\bigl( \p^1_\varepsilon+\p^2_\varepsilon+\p^3_\varepsilon \bigr)(x_0)$ and define
$$
\widetilde{\p}^3_\varepsilon=\p^3_\varepsilon-h_\varepsilon\,.
$$
Then $\p_\varepsilon:=(\p^1_\varepsilon, \p^2_\varepsilon, \widetilde{\p}^3_\varepsilon)\in\widetilde{H}^1(\g)$, $\p^i_\varepsilon\rightarrow\p^i$ in $H^1(\g^i)$ for $i=1,2$ and $\widetilde{\p}^3_\varepsilon\rightarrow\p^3$ in $H^1(\g_3)$. By the continuity of the quadratic form $\partial^2\ms\bigl( (u,\g); U \bigr)$ with respect to the $H^1$ convergence, we obtain that
$$
\partial^2\ms\bigl( (u,\g); U \bigr)[\p]=\lim_{\varepsilon\rightarrow0}\partial^2\ms\bigl( (u,\g); U \bigr)[\p_\varepsilon]\geq0\,.\\
$$

\emph{Step 2}. Let $\p\in\widetilde{H}^1(\g)$ such that $\p^i\in C^\infty(\bar{\g}_i)$ for all $i=1,2,3$.
The idea is the following: let $X\in C^2(\bar{\o};\R^2)$ be a vector field such that $X\cdot\nu_{\po}=0$ on $\po$, $X\equiv0$ in $\pdo\cup(\o\meno U)$ and suppose $X\cdot\nu^i=\p^i$ on $\g^i$. Then, by considering the flow
$(\Phi_t)_t$ generated by $X$, we have
\[
0\leq \frac{\mathrm{d}^2}{\mathrm{d} t^2}{\ms\bigl( (u,\g);U \bigr)}_{|t=0}=
           \partial^2\ms\bigl( (u,\g); U \bigr)[\p]\,,
\]
where the inequality follows by the local minimality property of $(u,\g)$.
Our strategy would be to work by approximation, \emph{i.e.}, we will construct a family of vector fields $(X^\varepsilon)_\varepsilon\subset C^2(\bar{\o};\R^2)$ satisfying $X^\varepsilon\cdot\nu_{\po}=0$ on $\po$, $X^\varepsilon\equiv0$ in $\pdo\cup(\o\meno U)$ and such that $X^\varepsilon\cdot\nu^i\rightarrow\p^i$ on $\g^i$
as $\varepsilon\rightarrow0$.

For every $x\in\bar{\g}_i$ define the vector
$$
Y(x):=\p^i(x)\nu_i(x)+b^i(x)\tau_i(x)\,,
$$
for some function $b^i\in C^1(\bar{\g}_i)$, that we have to choose.
The functions $b_i$'s will be chosen in order to satisfy the following conditions:
\begin{itemize}
\item[(i)] $Y(x)=\p^i(x)\nu^i(x)$ for $x\in\g_i\meno B_{\delta}(x_0)$, for some $\delta>0$,
\item[(ii)] $Y$ is well defined in $x_0$,
\item[(iii)] $DY(x_0)[\tau_1+\tau_2+\tau_3]=0$.
\end{itemize}
Thus, we require $b^i(x)\equiv0$ if $|x-x_0|\geq \delta$ in order to satisfy $(i)$ and, in order to obey also $(ii)$ and $(iii)$, we impose the following conditions
$$
\Bigl(\p^i + \frac{1}{2}\p^j - \frac{\sqrt{3}}{2}b^j  \Bigr)(x_0)=0\,,
$$
for all $i\neq j=1,2,3$, and
$$
\sum_{i=1}^3\bigl[ \nu_i D_{\g^i}\p^i + \tau_i D_{\g^i}b^i \bigr](x_0)=0\,,
$$
respectively, where in the derivation of the last one we have used the fact that $H_i(x_0)=0$, and hence $D_{\g^i}\nu_i(x_0)=D_{\g^i}\tau_i(x_0)=0$. Notice that it is possible to choose the $b^i$'s in such a way that all the above conditions are satisfied.

So, choose functions $b^i$'s satisfying the above conditions, and define the vector field $\bar{Y}:\g\cup\po\cup(\o\meno U)\rightarrow\R^2$ as follows:
$$
\bar{Y}:=
\left\{
\begin{array}{ll}
Y(x) & \text{ if } x\in\g\,,\\
a(x)\tau_{\partial\o}(x) & \text{ if } x\in\partial\o\,,\\
0 & \text{ in } \o\meno U\,,
\end{array}
\right.
$$
where $a\in C^1(\partial\o)$ is any function such that $a\equiv1$ in a neighborhood of $\partial \g \cap\partial\o$ and $a\equiv0$ in $\pdo\cup(\bar{U}\cap\partial\o)$.
Applying Lemma ~\ref{lem:3W} we can easily infer that $\bar{Y}\in C^1(\bar{\g}\cup\partial\o\cup(\o\meno U);\R^2)$.
So, by using Whitney's extension theorem, we can extend $\bar{Y}$ to a vector field
$\widetilde{Y}\in C^1(\bar{\o};\R^2)$.

Finally, we need to take care of the regularity of the vector field and of the tangential condition on $\po$.
To do so, by using convolutions, we can approximate $\widetilde{Y}$ with vector fields $\widetilde{X}^\varepsilon\in C^\infty(\bar{\o};\R^2)$ such that $\widetilde{X}^\varepsilon\rightarrow \widetilde{Y}$ in $C^1(\bar{\o};\R^2)$. Notice that $\supp\widetilde{X}^\varepsilon\subset\subset U'\meno\pdo$, where $U'\supset U$ is an admissible subdomain. Now define
$$
X^{\varepsilon}(x):=
\left\{
\begin{array}{ll}
\widetilde{X}^\varepsilon(x)-\chi\Bigl(\frac{s^2}{\delta^2}\Bigr)\bigl( (\widetilde{X}^\varepsilon\cdot\nu_{\partial\o})\nu_{\partial\o} \bigr)(y) & \text{ if } x=y+s\nu_{\partial\o}(y)\,,\, s<\delta,\\
\widetilde{X}^\varepsilon & \text{ otherwise}.\\
\end{array}
\right.
$$
In this way $X^{\varepsilon}\cdot\nu_{\partial\o}=0$ on $\partial\o$, and we still have that $X^{\varepsilon}\rightarrow\widetilde{Y}$ in $C^1(\bar{\o};\R^2)$.
In particular we have that $\p^\varepsilon_i:=X^{\varepsilon}\cdot\nu^i\rightarrow X\cdot\nu^i=\p^i$ on $\g^i$.
This allows to conclude. Indeed, let $(\Phi^\varepsilon)_t$ be the flow generated by the vector field $X^\varepsilon$, with $\Phi_0^\varepsilon=\id$, and
let $\g^\varepsilon_t:=\Phi^\varepsilon_t(\g)$ be the evolution of $\g$ through this flow.
Then, we have that
\begin{align*}
\partial^2\ms\bigl( (u,\g); U \bigr)[\p] &= 
   \lim_{\varepsilon\rightarrow0} \partial^2\ms\bigl( (u^\varepsilon,\g^\varepsilon); U \bigr)[\p^\varepsilon]\\
&= \lim_{\varepsilon\rightarrow0} \frac{\mathrm{d}^2}{\mathrm{d} t^2}{\ms\bigl( (u^\varepsilon,\g^\varepsilon);U \bigr)}_{|t=0} \geq 0\,,
\end{align*}
where in the first equality we have used the continuity of $\partial^2\ms$.
\end{proof}


\bigskip
\bigskip
\noindent
{\bf Acknowledgments.}
{The author wishes to thank Massimiliano Morini for introduced me to the study of this problem and for multiple helpful discussions we had during the preparation of this paper.

The author thanks SISSA and the Center for Nonlinear Analysis at Carnegie Mellon University for their support
during the preparation of the manuscript.
The research was partially supported by National Science Foundation under Grant No. DMS-1411646.
The results contained in the paper are part of the Ph.D. thesis of the author.
}


\bibliographystyle{siam} 
\bibliography{bibliografia}

\end{document}